\documentclass[11pt]{article}
\usepackage{amsthm, amsmath, amssymb,mathrsfs, enumerate, fullpage, mathtools, extpfeil}
\usepackage{colonequals}
\usepackage{microtype}
\usepackage{tocloft}
\allowdisplaybreaks 
\mathtoolsset{showonlyrefs=true} 

\usepackage{xcolor}
\usepackage{comment}
\definecolor{darkerGreen}{RGB}{0,195,0} 
\definecolor{darkerRed}{RGB}{240,0,0}
\usepackage[colorlinks=true, linkcolor=darkerRed, citecolor=darkerGreen]{hyperref}
\numberwithin{equation}{section}

\theoremstyle{plain}
\newtheorem{thm}{Theorem}[section]
\newtheorem{theorem}[thm]{Theorem}
\newtheorem{corollary}[thm]{Corollary}
\newtheorem{lemma}[thm]{Lemma}
\newtheorem{proposition}[thm]{Proposition}

\theoremstyle{remark} 
\newtheorem{remark}[thm]{Remark}
\newtheorem{example}[thm]{Example}

\theoremstyle{definition}
\newtheorem{definition}[thm]{Definition}
\newtheorem{assumption}[thm]{Assumption}

\newcommand{\BR}{{\mathbb {R}}}
\newcommand{\BZ}{{\mathbb {Z}}}
\newcommand{\CB}{{\mathcal {B}}}

\newcommand{\SF}{\mathscr{F}}
\newcommand{\SG}{\mathscr{G}}
\newcommand{\dd}{{\,\mathrm{d}}}

\begin{document}

\title{Barron Space Representations for Elliptic PDEs with Homogeneous Boundary Conditions}
\author{Ziang Chen and Liqiang Huang}
\date{}
\maketitle
\begin{abstract}
We study the complexity of approximating high-dimensional second-order elliptic PDEs with homogeneous boundary conditions on the unit hypercube using Barron spaces. Under suitable Barron assumptions on the coefficients and forcing term, we prove that the solutions can be approximated to any prescribed accuracy \(\varepsilon>0\) by two-layer neural networks whose widths and relevant parameters are bounded by \(\mathcal{O}\bigl(d^{C|\log\varepsilon|}\bigr)\). Consequently, we identify a class of elliptic PDEs that can be approximated by shallow neural networks without suffering from the curse of dimensionality.
\end{abstract}
\tableofcontents

\section{Introduction}
High-dimensional partial differential equations (PDEs) arise in a wide range of applications, 
including physics, finance, and control. 
Classical numerical schemes, such as finite difference and finite element methods, 
become computationally intractable as the spatial dimension \(d\) increases, 
due to the so-called \emph{curse of dimensionality (CoD)}: 
to achieve an accuracy of \(\varepsilon\), 
the computational cost scales as \(\mathcal{O}(\varepsilon^{-d})\), exponentially in $d$.

To overcome this barrier, a variety of deep learning-based approaches have been developed in recent years~\cite{EHJ17,HJE18,Sir18,EYu18,HZE19,RPK19,HS20,PSMF20,ZhouHL21,NakaTGK21,EJA22,LL24,KT25}. As an example, the \emph{Deep Ritz method (DRM)} proposed by E and Yu~\cite{EYu18} is a powerful approach for solving PDEs that admit a variational formulation. In this method, the PDE solution is represented by a deep neural network, and the network parameters are determined by minimizing the associated energy functional, with the integrals evaluated via Monte Carlo sampling. Since only first-order derivatives of the network output are required, this approach is particularly efficient for PDEs involving higher-order derivatives. Additionally, the \emph{Physics-Informed Neural Network (PINN)} developed by Raissi et al.~\cite{RPK19} embeds the PDE and boundary conditions into the loss function by penalizing violations at selected collocation points. While this framework applies to a broader class of PDEs, it typically requires computing higher-order derivatives of the network output, which can reduce efficiency for PDEs with high-order terms.

Despite differences in formulation, both DRM and PINNs share the goal of learning PDE solutions with neural networks. Their practical performance hinges on several theoretical factors: the representational capacity of the network architecture, the design and convergence of the training algorithm, and the generalization ability of the trained model to unseen inputs~\cite{SDK20,LuLuWang21,LuLu22,DTM24,JSC26}. These considerations motivate three fundamental theoretical questions for deep-learning–based PDE methods:
\begin{itemize}
\item \emph{Approximation}: What is the minimal architectural complexity (e.g., depth, width) required to approximate a target function to a prescribed accuracy?
\item \emph{Optimization}: For a given architecture, how can we design training algorithms with provable convergence to an optimal or near-optimal solution?
\item \emph{Generalization}: Under what conditions does the trained network generalize effectively, achieving low error on unseen data?
\end{itemize}

In this work, we focus on the first question. 
A prominent framework for addressing this problem is based on \emph{Barron functions}. 
In the seminal work~\cite{Barron93}, Barron showed that a real-valued function \(f\colon\mathbb{R}^d\to\mathbb{R}\) can be approximated in the \(L^2\)-norm by a two-layer neural network, provided that its Fourier transform \(\widehat{f}\) satisfies
\begin{equation}
\label{oB}
\int_{\mathbb{R}^d}|\xi|\cdot |\widehat{f}(\xi)|\,\dd\xi<\infty.
\end{equation}
This condition is now commonly referred to as finiteness of the \emph{Barron norm}. The approximating network takes the form
\begin{equation}
\label{eq:2NN}
\frac{1}{k}\sum_{i=1}^k
a_i\sigma\bigl(w_i^\top x+b_i\bigr).
\end{equation}
If the Barron norm is uniformly bounded, the resulting approximation error bound is independent of the input dimension \(d\). This result provides a rigorous theoretical foundation for the ability of neural networks 
to circumvent the CoD in high-dimensional function approximation; that is, the required number of neurons does not grow exponentially with the input dimension. Following~\cite{Barron93}, there have been multiple variants and generalizations of Barron functions with dimension-independent approximation results in a similar spirit, see e.g.~\cite{Jason18, E, EWoj22, SIEGELXu21, SX23, SW24, CLTakase26}. 

This observation suggests a new paradigm for studying high-dimensional PDEs:
\begin{quote}
\emph{If the PDE data ensure that the solution admits an accurate Barron approximation without suffering from the CoD, then the solution can be efficiently approximated in high dimensions by a two-layer neural network.}
\end{quote}
This perspective was first realized by Chen et al.~\cite{Chen21}, who developed a Barron function framework for approximating second-order elliptic PDEs on \(\mathbb{R}^d\). Motivated by~\cite{MLR21}, their iterative schemes converge exponentially while controlling the Barron norm. Marwah et al.~\cite{Mar23} extended this approach to Dirichlet problems on the unit hypercube under additional regularity assumptions on the coefficients; see the discussion following Theorem~\ref{mainthm} in Section~\ref{subsecmain}. By contrast, general variable-coefficient Neumann problems on bounded domains remain unexplored.

\subsection{Contributions}
\label{ourcontri}
We study the approximation of solutions to second-order elliptic PDEs on the \(d\)-dimensional unit hypercube \(\Omega\coloneqq (0,1)^d\subseteq\mathbb{R}^d\), where \(d\geq 2\). Specifically, we consider the following two prototypical elliptic problems with homogeneous boundary conditions:
\begin{itemize}
    \item \textbf{The Dirichlet Problem:}
    \begin{equation}
    \label{Diri}
    -\nabla \cdot  (A(x) \nabla u)  + c(x) u = f(x) \text{ in } \Omega,\quad u = 0 \text{ on } \partial \Omega,
    \end{equation}

    \item \textbf{The Neumann Problem:}
    \begin{equation}
    \label{Newu}
    -\nabla \cdot  (A(x) \nabla u)  + c(x) u = f(x) \text{ in } \Omega,\quad A(x) \nabla u \cdot \nu = 0 \text{ on } \partial \Omega,
    \end{equation}
\end{itemize}
where \( A(x) \) is a real matrix-valued function, and \( \nu \) denotes the outward unit normal vector on \( \partial \Omega \). Our main results are stated informally below; precise formulations are given in Theorems~\ref{mainthm} and~\ref{mainthm2}.

\begin{theorem}[Main Results, Informal Version]
Suppose that the coefficients \(A(x)\), \(c(x)\) and the forcing term \(f(x)\) in problems~\eqref{Diri} and~\eqref{Newu} belong to suitable Barron spaces. Then for every \(0<\varepsilon<1/2\), the weak solution \(u^*\) of either problem can be approximated with error at most \(\varepsilon\) in the \(H^1(\Omega)\)-norm by a two-layer neural network of the form
\[
\frac{1}{k}\sum_{i=1}^k a_i\cos\bigl(w_i^\top x+b_i\bigr)
\quad\text{or}\quad
\gamma+\sum_{i=1}^k a_i\operatorname{ReLU}\bigl(w_i^\top x+b_i\bigr).
\]
Moreover, the network width satisfies \(k=\mathcal{O}\bigl(d^{C_1|\log\varepsilon|}\bigr)\), and its parameters are bounded by \(\mathcal{O}\bigl(d^{C_2|\log\varepsilon|}\bigr)\). The positive constants \(C_1\) and \(C_2\) depend only on the Barron norms of the problem data.
\end{theorem}

We briefly outline the strategy below, with full details and precise notation provided in the subsequent section.
\begin{enumerate}[Step 1.]
	\item For each of the problems~\eqref{Diri} and~\eqref{Newu}, we construct a Sobolev gradient descent scheme based on its variational formulation:
\[
u_{t+1}=u_t-\alpha\mathcal{G}(u_t),
\]
where \(\mathcal{G}\) is the Sobolev gradient of the associated energy functional and \(\alpha>0\) is the step size. We show that the scheme converges exponentially to the solution \(u^*\) in the \(H^1(\Omega)\)-norm:
\[
\|u_t-u^*\|_{H^1(\Omega)}
\lesssim \beta^t
\quad\text{for all }t\geq0,
\]
where \(0<\beta<1\) depends only on the PDE coefficients; see~\eqref{setp1D} and~\eqref{setp1N}. The iterative schemes proposed in~\cite{Chen21,Mar23} can be viewed as special cases of this Sobolev gradient descent framework.
    
	\item We derive a recursive estimate for the Barron norm of \(u_{t+1}\) in terms of that of \(u_t\), with a one-step growth factor of order \(\mathcal{O}(d^2)\); see Theorem~\ref{thm:baroonnormut+1}. Consequently, every iterate \(u_t\) remains in the corresponding Barron space.
	\item For cosine networks, the Barron approximation theory developed in~\cite{E} shows that every Barron function \(g\) admits a two-layer cosine network satisfying the following estimate; see Theorem~\ref{thm:nueraltobarron}:
\[
\left\|
\frac{1}{k}\sum_{i=1}^k a_i\cos\bigl(w_i^\top x+b_i\bigr)-g
\right\|_{H^1(\Omega)}
\leq
\frac{\|g\|_{\mathcal{B}(\Omega)}}{\sqrt{k}},
\]
where \(\|g\|_{\mathcal{B}(\Omega)}\) denotes the Barron norm of \(g\). For networks with \(\operatorname{ReLU}\) activation, we extend the strategy of~\cite{LuLuWang21} and obtain the following estimate; see Theorem~\ref{thmllw21}:
\[
\left\|
\gamma+\sum_{i=1}^k a_i\operatorname{ReLU}\bigl(w_i^\top x+b_i\bigr)-g
\right\|_{H^1(\Omega)}
\lesssim
\frac{\|g\|_{\CB(\Omega)}}{\sqrt{k}}.
\]
\end{enumerate}
We give an example illustrating how Barron functions help overcome the curse of dimensionality.

\begin{example}
Let $\mathcal{P}_D$
denote the class of elliptic Dirichlet problems of the form~\eqref{Diri} satisfying the following conditions:
\begin{enumerate}[(1)]
    \item The spatial dimension \(d\geq2\) is arbitrary.
    
    \item The eigenvalues of $A(x)$ and the values of $c(x)$ lie in the interval $[1,1.1]$.
    
    \item The relevant Barron norms of $A$, $c$, and $f$ are bounded by $1.2$.
    
    \item The source term satisfies
    \(\|f\|_{H^{-1}(\Omega)}\leq 0.1\).
\end{enumerate}
The precise definitions of these quantities are given in Section~\ref{sec:2}.
Our main results imply that, for every problem in $\mathcal{P}_D$ and every
$0<\varepsilon<0.5$, there exists a two-layer cosine neural network
\[
    N_k(x)=\frac{1}{k}\sum_{i=1}^k
    a_i\cos\bigl(w_i^\top x+b_i\bigr)
\]
such that
\[
    \|N_k-u_D^*\|_{H^1(\Omega)}\leq\varepsilon.
\]
Moreover, the network width satisfies
\[
    k\leq 2d^{11|\log\varepsilon|},
\]
while its parameters satisfy
\[
    |a_i|\leq 2d^{11|\log\varepsilon|},\quad
    w_i\in\pi\mathbb{Z}^d, \quad
    \frac{1}{k}\sum_{i=1}^k |a_i|\bigl(1+|w_i|^2\bigr)
    \leq 2d^{11|\log\varepsilon|}.
\]
Thus for every fixed \(\varepsilon\), both the network width and the weighted
parameter bound grow at most polynomially in \(d\). Hence this class of
variable-coefficient elliptic Dirichlet problems can be approximated by
shallow neural networks without exponential dependence on the dimension.
Analogous polynomial-in-\(d\) bounds hold for two-layer
\(\operatorname{ReLU}\) networks.
\end{example}

\subsection{Related Work}
Beyond approximating PDE solutions by Barron functions, the Barron framework raises another natural question:
\begin{quote}
\emph{If the PDE coefficients are sufficiently regular in the Barron sense, does the solution also belong to the Barron class?}
\end{quote}
This viewpoint is often referred to as the regularity problem. It was first explored in~\cite{LuLuWang21,LuLu22,EWreg22}, where the authors investigated the Poisson equation and the Schr\"{o}dinger equation on the unit hypercube with Neumann boundary conditions. Subsequently, Chen et al.~\cite{CLLZ23} implemented this perspective for the stationary Schr\"{o}dinger equation on \( \mathbb{R}^d \). Regularity results for the Hamilton--Jacobi--Bellman equation in the whole space were established in~\cite{Feng25}, and for the electronic Schr\"{o}dinger equation in~\cite{Harry25}.

Apart from the Barron function-based approach, another influential deep learning methodology was proposed by E et al.~\cite{EHJ17,HJE18}, who reformulated certain parabolic PDEs as backward stochastic differential equations (BSDEs), and further interpreted these BSDEs as stochastic control problems. These problems, in turn, can be viewed as instances of model-based reinforcement learning. This SDE-based framework was later extended by Sirignano and Spiliopoulos~\cite{Sir18}. We refer to~\cite{AYT22, M18,LL24,LVR24} for additional applications of SDE-based methods to high-dimensional PDEs.

Alternative deep learning-based approaches have also been developed, typically relying on strong structural assumptions on the coefficients of the underlying PDEs. For instance,~\cite{PhilippL21, GHJv23} investigate classes of parabolic and Poisson equations that admit stochastic representations, such as those derived from the Feynman--Kac formula. In contrast,~\cite{MLR21} proposes neural network approximation conditions on the coefficients of second-order elliptic PDEs.

\subsection{Organization}
Section~\ref{sec:2} introduces the Barron spaces and their basic properties before stating the main results. Section~\ref{sec:proofsofmain} presents the proofs based on the three-step strategy outlined in Section~\ref{ourcontri}.

\section{Barron Spaces and Main Results}
\label{sec:2}
In this section, we present our main results. 
To ensure the existence and uniqueness of weak solutions to PDEs~\eqref{Diri} and~\eqref{Newu}, we impose the following assumption on their coefficients. All functions in this paper are real-valued unless otherwise noted.

\begin{assumption}[Coefficients of Elliptic PDEs]
\label{A1}
\hfill
\begin{enumerate}[(1)]
	\item 
The matrix-valued function \( A(x) = (A_{ij}(x))_{1 \leq i,j \leq d} \) defined on \( \Omega \) is symmetric and \emph{uniformly elliptic}. By uniform ellipticity, we mean that there exists a constant \( a_{\min} > 0 \) such that
\[
v^{\top} A(x) v \geq a_{\min} | v |^2 \quad \text{for all } x \in \Omega \text{ and } v \in \mathbb{R}^d,
\]
where \( |v| \) denotes the standard Euclidean norm in \( \mathbb{R}^d \).  
Moreover, the \emph{operator norm} of \( A(x) \), defined as
\[
\|A\|_{\mathrm{op}} \coloneqq \sup_{x \in \Omega} \sup_{v \in \mathbb{R}^d \setminus \{0\}} \frac{|A(x)v|}{|v|},
\]
is finite, and we denote \( a_{\max} \coloneqq \|A\|_{\mathrm{op}}\). 
 
\item The scalar coefficient function \( c(x) \) defined on \( \Omega \) satisfies 
\[
0 < c_{\min} \leq c(x) \leq c_{\max} < \infty \quad \text{for all } x \in \Omega.
\]

\item The source term \( f(x) \) belongs to \( L^2(\Omega) \).
\end{enumerate}
\end{assumption}

We refer to \cite{Evans10, Lieb13} for a comprehensive treatment of weak solutions in Sobolev spaces, and provide a brief summary here for completeness. 
A function \( u_D^* \in H_0^1(\Omega) \) is said to be a weak solution of the Dirichlet problem~\eqref{Diri} if it satisfies
\begin{equation}
\label{uDprop}
\int_{\Omega} \left(\nabla u_D^* \cdot A\nabla v + c u_D^* v - f v \right) \dd x = 0 \quad \text{for all } v \in H_0^1(\Omega).
\end{equation}
A function \( u_N^* \in H^1(\Omega) \) is said to be a weak solution of the Neumann problem~\eqref{Newu} if it satisfies
\begin{equation}
\label{uNprop}
\int_{\Omega} \left( \nabla u_N^* \cdot A\nabla v + c u_N^* v - f v \right) \dd x = 0 \quad \text{for all } v \in H^1(\Omega).
\end{equation}
Assumption~\ref{A1} and the Lax--Milgram theorem~\cite[Theorem 6.2.1]{Evans10} ensure the existence and uniqueness of weak solutions \( u_D^* \in H_0^1(\Omega) \) and \( u_N^* \in H^1(\Omega) \) to~\eqref{Diri} and~\eqref{Newu}, respectively.

\subsection{Barron Spaces on the Unit Hypercube}
\label{subsecBarron}
To define the four types of Barron functions considered in this paper, we begin by introducing in Theorem~\ref{bases} four families of functions in \( L^2(\Omega) \), which serve as building blocks for representing elements in this space. Some of these families form bases of \( L^2(\Omega) \). The justification is deferred to Appendix~\ref{pfbases}.
\begin{theorem}
\label{bases}
Let \( \mathbb{N} \coloneqq \{0, 1, 2, \dots\} \) denote the set of natural numbers.  
For \( k = (k_1, \dots, k_d) \in \mathbb{N}^d \), \( \omega \in \mathbb{Z}^d \), and \( x = (x_1, \dots, x_d) \in \Omega \),  we define the following four families of functions in \( L^2(\Omega) \):
\begin{enumerate}[(1)]
    \item \( \left\{ e^{i \pi \omega^{\top} x} : \omega \in \mathbb{Z}^d \right\} \),
    \item \( \left\{ S_k(x) =\prod_{i=1}^{d} \sin(\pi k_i x_i) : k \in \mathbb{N}^d,\ k_i \neq 0  \text{ for all } i \right\} \),
    \item \( \left\{ C_k(x) = \prod_{i=1}^{d} \cos(\pi k_i x_i) : k \in \mathbb{N}^d \right\} \),
    \item \( \left\{ M_{k,i,j}(x) = \sin(\pi k_i x_i)\sin(\pi k_j x_j)
\prod_{m\neq i,j}\cos(\pi k_m x_m) : k \in \mathbb{N}^d,\ k_i,k_j\neq 0 \right\} \), where \( 1 \leq i \ne j \leq d \) are fixed. 
\end{enumerate}
Every function in \( L^2(\Omega) \) admits an \( L^2 \)-expansion in family~(1) with coefficients in \(\mathbb{C}\), though the representation may be non-unique. Families~(2)--(4) form orthogonal bases of \( L^2(\Omega) \) with coefficients in \(\mathbb{R}\).
\end{theorem}

\begin{definition}[Barron Norms and Barron Functions]
Suppose that \(g\in L^2(\Omega)\), and adopt the convention \(0^0=1\).
	\begin{enumerate}[(1)]
	    \item   We say that \( g \) is an \emph{\( e \)-Barron function} of weight \( n \geq 0 \) if its weighted \emph{\( e \)-Barron norm} 
\[
\|g\|_{\mathcal{B}_e^n(\Omega)} \coloneqq 
\inf \left\{ 
\sum_{\omega \in \mathbb{Z}^d} |a_{\omega}| \left( 1 + \pi^n |\omega|^n \right)
: 
g(x) = \sum_{\omega \in \mathbb{Z}^d} a_{\omega} e^{i \pi \omega^{\top} x} \; \text{in } L^2(\Omega)
\right\}.
\]
is finite. The space \( \mathcal{B}_e^n(\Omega) \) is referred to as the \emph{\( e \)-Barron space} of weight \( n \).

		\item We say that \( g \) is an \emph{\( s \)-Barron function} of weight \( n \geq 0 \)  if it admits the sine basis expansion
\[
g(x) = \sum_{k \in \mathbb{N}^d} a_k S_k(x)\quad \text{in } L^2(\Omega),
\]
with finite weighted \emph{\( s \)-Barron norm}
\[
\|g\|_{\mathcal{B}_s^n(\Omega)} \coloneqq \sum_{k \in \mathbb{N}^d} |a_k| \left(1 + \pi^n |k|^n \right)
\]
The space \( \mathcal{B}_s^n(\Omega) \) is referred to as the \emph{\( s \)-Barron space} of weight \( n \).
In sine expansions, we adopt the conventions that \(a_k=0\) whenever at least one component of \(k\) is zero.

\item We say that \( g \) is a \emph{\( c \)-Barron function} of weight \( n \geq 0 \)  if it admits the cosine basis expansion
\[
g(x) = \sum_{k \in \mathbb{N}^d} a_k C_k(x)\quad \text{in } L^2(\Omega),
\]
with finite weighted \emph{\( c \)-Barron norm}
\[
\|g\|_{\mathcal{B}_c^n(\Omega)} \coloneqq \sum_{k \in \mathbb{N}^d} |a_k| \left(1 + \pi^n |k|^n \right).
\]
The space \( \mathcal{B}_c^n(\Omega) \) is referred to as the \emph{\( c \)-Barron space} of weight \( n \).

\item 
We say that \( g \) is an \emph{\((i,j)\)-mixed Barron function} of weight \( n \geq 0 \) (for fixed indices \( i \neq j \)) if it admits a mixed basis expansion
 \[
g(x) = \sum_{k \in \mathbb{N}^d} a_k M_{k,i,j}(x)\quad \text{in } L^2(\Omega),
\]
with finite weighted \emph{\((i,j)\)-mixed Barron norm}
\[
\|g\|_{\mathcal{B}_{i,j}^n(\Omega)} \coloneqq \sum_{k \in \mathbb{N}^d} |a_k| \left(1 + \pi^n |k|^n \right).
\]
The space \( \mathcal{B}_{i,j}^n(\Omega) \) is referred to as the \emph{\((i,j)\)-mixed Barron space} of  weight \( n \).
In mixed expansions, we adopt the conventions that \(a_k=0\) whenever \(k_i=0\) or \(k_j=0\).
	\end{enumerate}
\end{definition}

The original definition of the Barron norm for functions on \( \mathbb{R}^d \) was introduced in~\cite{Barron93}, as given in equation~\eqref{oB}. In contemporary literature, the version of the Barron norm defined via the Fourier transform, used in, e.g.,~\cite{SIEGELXu21, Mar23, Feng25}, is commonly referred to as the \emph{spectral Barron norm}. In contrast, definitions that avoid the use of the Fourier transform, such as those based on probability measures (e.g.,~\cite{Chen21, E, SX23}), are typically referred to simply as the \emph{Barron norm}.
Our formulation is inspired by~\cite{LuLuWang21, LuLu22, Mar23}, where cosine basis functions were used in~\cite{LuLuWang21, LuLu22} and exponential basis functions in~\cite{Mar23}. In this work, we introduce the sine-based Barron norm and the \((i,j)\)-mixed Barron norm. The sine basis arises naturally when considering functions that vanish on the boundary, while the mixed norm is introduced for certain technical reasons (cf.~Lemma~\ref{lem:Ajterms}). 
We also note that if the coefficient matrix \( A(x) \) is diagonal, then the mixed Barron norm is not needed. This observation already covers many important families of PDEs, such as the stationary reaction--diffusion equations and static Schr\"{o}dinger equations.

Note that if \( g \) is an \( s \)-Barron function, then any expansion with respect to the sine basis can be directly rewritten in terms of the exponential basis:
\[
    g(x)  = \sum_{k \in \mathbb{N}^d} a_k S_k(x)
= \sum_{k \in \mathbb{N}^d} \frac{a_k}{(2i)^d} \prod_{j=1}^{d} \left( e^{i \pi k_j x_j} - e^{-i \pi k_j x_j} \right)  = \sum_{k \in \mathbb{N}^d} \frac{a_k}{(2i)^d} \sum_{\epsilon\in\{\pm 1\}^d} \prod_{j=1}^d \epsilon_j \cdot e^{i\pi (\epsilon\circ k)^\top x},
\]
where $\epsilon\circ k$ is the Hadamard product of $\epsilon$ and $k$.
This observation implies that
\[
\|g\|_{\mathcal{B}_e^n(\Omega)} 
\leq 
\frac{1}{2^d}
\sum_{k \in \mathbb{N}^d}\sum_{\epsilon\in\{\pm 1\}^d}
 |a_k| \left( 1 + \pi^n |\epsilon\circ k|^n \right)
= \|g\|_{\mathcal{B}_s^n(\Omega)}.
\]
A similar argument applies to other types of Barron functions, yielding the following result.

\begin{proposition}
\label{prop:barron}
Let \( \widetilde{\mathcal{B}}^n(\Omega) \) denote any of the spaces \( \mathcal{B}_s^n(\Omega) \), \( \mathcal{B}_c^n(\Omega) \), or \( \mathcal{B}_{i,j}^n(\Omega) \).  
If \( g \in \widetilde{\mathcal{B}}^n(\Omega) \), then
\[
\|g\|_{\mathcal{B}_e^n(\Omega)} \leq \|g\|_{\widetilde{\mathcal{B}}^n(\Omega)}.
\]
In particular, the following inclusions hold:
\[
\mathcal{B}_s^n(\Omega),\; \mathcal{B}_c^n(\Omega),\; \mathcal{B}_{i,j}^n(\Omega) \subseteq \mathcal{B}_e^n(\Omega).
\]
\end{proposition}

\subsection{Main Approximation Results}
\label{subsecmain}
Our main results are established under the following assumptions on the coefficients appearing in equations~\eqref{Diri} and~\eqref{Newu}.
\begin{assumption}[Coefficients of the Dirichlet Problem]
\label{A2}
\hfill
\begin{enumerate}[(1)]
	\item The entries of the coefficient matrix \( A(x) = (A_{ij}(x))_{1 \leq i,j \leq d} \) satisfy the following: \( A_{ii}(x) \in \mathcal{B}_c^1(\Omega) \) for all \( i \), and \( A_{ij}(x) \in \mathcal{B}_{i,j}^1(\Omega) \) for all \( i \ne j \). We define
\[
\ell_{A} \coloneqq \max\left\{ \max_{1 \leq i \leq d} \|A_{ii}\|_{\mathcal{B}_c^1(\Omega)},\ \max_{1 \leq i \ne j \leq d} \|A_{ij}\|_{\mathcal{B}_{i,j}^1(\Omega)} \right\}.
\]
\item The scalar coefficient \( c(x) \) lies in \( \mathcal{B}_c^2(\Omega) \), and we denote its Barron norm by \( \ell_{c} \coloneqq \|c\|_{\mathcal{B}_c^2(\Omega)} \).  
\item The source function \( f(x) \) belongs to \( \mathcal{B}_s^0(\Omega) \), with norm \( \ell_{f,D} \coloneqq \|f\|_{\mathcal{B}_s^0(\Omega)} \).
\end{enumerate}
\end{assumption}

\begin{assumption}[Coefficients of the Neumann Problem]
\label{A2'}
Same as Assumption~\ref{A2}, except that \(f\in\mathcal{B}_c^0(\Omega)\), with \(\ell_{f,N}\coloneqq\|f\|_{\mathcal{B}_c^0(\Omega)}\).
\end{assumption}

For consistency in the subsequent analysis, we equip the Hilbert space \( H_0^1(\Omega) \) with the inner product inherited from the standard inner product on \( H^1(\Omega) \),
\[
(u,v)_{H^1(\Omega)}
\coloneqq
\int_{\Omega}
\nabla u\cdot\nabla v + uv \,\mathrm{d}x,
\quad
u,v\in H_0^1(\Omega).
\]
Although the norm induced by this inner product differs from the usual one on \( H_0^1(\Omega) \), which is based solely on the gradient term, they are equivalent by Poincar\'{e}'s inequality (cf.~\cite[Theorem~5.6.3]{Evans10}). This choice has no effect on our results.

We now state the first main result of the paper. Further quantitative details are provided in Remark~\ref{rmk:para}, and the proof is deferred to the next section.

\begin{theorem}[Main Result 1]
\label{mainthm}
Let \(u_D^*\in H_0^1(\Omega)\) and \(u_N^*\in H^1(\Omega)\) be the weak
solutions of the Dirichlet problem~\eqref{Diri} and the Neumann
problem~\eqref{Newu}, respectively.

\begin{enumerate}[(1)]
    \item Suppose that Assumptions~\ref{A1} and~\ref{A2} hold. Then, for every
\(0<\varepsilon<1/2\), there exist a positive integer \(k\) and a two-layer
cosine neural network
\begin{equation}
\label{eq:NkNN}
N_k(x)
=
\frac{1}{k}\sum_{i=1}^k
a_i\cos\bigl(w_i^\top x+b_i\bigr),
\end{equation}
whose parameters and \(e\)-Barron norm satisfy
\begin{equation}
\label{eq:Nksat}
|a_i|\leq M_D,\quad
w_i\in\pi\BZ^d,\quad
\frac{1}{k} \sum_{i=1}^k |a_i|\left(1+|w_i|^2\right)\leq M_D,\quad
\|N_k\|_{\mathcal{B}_e^2(\Omega)}\leq M_D,
\end{equation}
for \(i=1,\ldots,k\), and such that
\[
\|N_k-u_D^*\|_{H^1(\Omega)}
\leq \varepsilon.
\]

    \item Suppose Assumptions~\ref{A1} and~\ref{A2'} hold. Then, for every
\(0<\varepsilon<1/2\), there exists a two-layer cosine neural network of
the same form as~\eqref{eq:NkNN}, satisfying~\eqref{eq:Nksat} with
\(D\) replaced by \(N\), such that
\[
\|N_k-u_N^*\|_{H^1(\Omega)}
\leq \varepsilon.
\]
\end{enumerate}
In both cases, the network width \(k\) can be chosen such that \(k=\mathcal{O}\bigl(d^{C_1|\log\varepsilon|}\bigr)\), and the corresponding parameter bound \(M_X\), with \(X\in\{D,N\}\), satisfies \(M_X=\mathcal{O}\bigl(d^{C_2|\log\varepsilon|}\bigr)\), where \(C_1,C_2>0\) depend only on the ellipticity constants and the relevant Barron and Sobolev norms of the coefficients and the forcing term.
\end{theorem}

\begin{remark}[Quantitative Estimates for Cosine Networks]
\label{rmk:para}
The network widths in Theorem~\ref{mainthm} can be chosen to satisfy
\[
k_D
\leq
\left\lfloor 4M_D^2\varepsilon^{-2}\right\rfloor+1,
\qquad
k_N
\leq
\left\lfloor 4M_N^2\varepsilon^{-2}\right\rfloor+1,
\]
for the Dirichlet and Neumann cases, respectively, where
\[
M_X
\coloneqq
\frac{\alpha_{\ast}\ell_{f,X}
\bigl(p(d)^{T_X}-1\bigr)}
{2\bigl(p(d)-1\bigr)},
\
X\in\{D,N\},
\qquad
p(d)
=
\frac{1+\pi}{\pi}\alpha_{\ast}\ell_A d^2
+\alpha_{\ast}\ell_c+1,
\]
and
\[
 T_D=\left\lceil
 \frac{\log_+\!\left(2\|f\|_{H^{-1}(\Omega)}/
 (\varepsilon\lambda_{\min})\right)}{|\log\beta_*|}
 \right\rceil,\qquad
 T_N=\left\lceil
 \frac{\log_+\!\left(2\|f\|_{(H^1(\Omega))^*}/
 (\varepsilon\lambda_{\min})\right)}{|\log\beta_*|}
 \right\rceil.
\]
Here
\[
 \alpha_*=\frac{\lambda_{\min}}{\lambda_{\max}^2},\quad
 \beta_*=\left(1-\frac{\lambda_{\min}^2}{\lambda_{\max}^2}\right)^{1/2},\quad
 \lambda_{\min}=\min\{a_{\min},c_{\min}\},\quad
 \lambda_{\max}=\max\{a_{\max},c_{\max}\}.
\]
When \(\beta_*=0\), we set \(T_D=T_N=1\).
For simplicity, we use the above definition of \(p(d)\). Slightly smaller choices are possible; see the proof of
Theorem~\ref{thm:baroonnormut+1}. Nevertheless, this quantity retains
the same quadratic scaling in \(d\).
\end{remark}

Our result is stronger and more general than those in~\cite{Chen21,Mar23} in the following aspects:
\begin{enumerate}[(1)]
    \item Unlike~\cite{Chen21}, we do not require the solution of the Dirichlet problem~\eqref{Diri} to be the restriction of a solution defined on \(\mathbb{R}^d\). We also establish the explicit parameter bounds in~\eqref{eq:Nksat}, which were not provided in~\cite{Chen21};
    \item The work~\cite{Mar23} claims a comparable network size. However, its derivation relies on the assumption that only finitely many coefficients in the Barron expansion are nonzero, an assumption that is not required in our analysis;
    \item In addition to considering only the Dirichlet problem as in~\cite{Chen21,Mar23}, 
we also provide a result for the Neumann problem.
\end{enumerate}

Theorem~\ref{mainthm} concerns neural networks with the cosine activation, a representative periodic activation function whose effectiveness has been empirically demonstrated in recent works~\cite{HMRM20,MTS21,Benb22,RRG25}. We next establish an analogous result for the more commonly used \(\operatorname{ReLU}\) activation by adapting the approximation techniques of~\cite{LuLuWang21}. Explicit parameter choices and quantitative width bounds are provided in the remark following the theorem, while the proof is deferred to the next section.
\begin{theorem}[Main Result 2]
\label{mainthm2}
Let \(u_D^*\in H_0^1(\Omega)\) and \(u_N^*\in H^1(\Omega)\) be the weak
solutions of the Dirichlet problem~\eqref{Diri} and the Neumann
problem~\eqref{Newu}, respectively.

\begin{enumerate}[(1)]
    \item 
Suppose that Assumptions~\ref{A1} and~\ref{A2} hold. Then, for every
\(0<\varepsilon<1/2\), there exists a two-layer ReLU neural network of the form
\begin{equation}
    \label{eq:Rk}
    R_k(x)
    =
    \gamma+\sum_{i=1}^k
    a_i\operatorname{ReLU}\bigl(w_i^\top x+b_i\bigr),
\end{equation}
whose coefficients satisfy
\begin{equation}
    \label{eq:coffebd}
    |\gamma|
    \leq
    2M_D,
    \quad
    \sum_{i=1}^{k}|a_i|
    \leq
    4\sqrt{d}M_D,
    \quad
    |w_i|=1,
    \quad
    |b_i|\leq\sqrt{d},
\end{equation}
for \(i=1,\ldots,k\), and such that
\[
    \| R_k-u_D^*\|_{H^1(\Omega)}
    \leq \varepsilon.
\]
    
    \item Suppose that Assumptions~\ref{A1} and~\ref{A2'} hold. Then, for every
\(0<\varepsilon<1/2\), there exists a two-layer ReLU network \(R_k\)
satisfying~\eqref{eq:Rk} and~\eqref{eq:coffebd}, with \(D\) replaced
by \(N\), such that
\[
    \lVert R_k-u_N^*\rVert_{H^1(\Omega)}\leq\varepsilon.
\]
\end{enumerate}
In both cases, the network width \(k\) can be chosen such that \(k=\mathcal{O}\bigl(d^{C_1|\log\varepsilon|}\bigr)\), and the corresponding parameter bound \(M_X\), with \(X\in\{D,N\}\), satisfies \(M_X=\mathcal{O}\bigl(d^{C_2|\log\varepsilon|}\bigr)\),
where \(C_1, C_2>0\) depend only on the ellipticity constants and the relevant
norms of the PDE coefficients and the forcing term.
\end{theorem}
\begin{remark}[Quantitative Width Bounds for $\operatorname{ReLU}$ Networks]
The network widths in Theorem~\ref{mainthm2} can be chosen to satisfy
\[
\begin{aligned}
    k_D
    &\le
    \left\lfloor
        4q(d)M_D^2\varepsilon^{-2}
    \right\rfloor+1,
    &\quad
    k_N
    &\le
    \left\lfloor
        4q(d)M_N^2\varepsilon^{-2}
    \right\rfloor+1,
\end{aligned}
\]
for the Dirichlet and Neumann cases, respectively.
Here, \(q(d)=64d^2+48d+4\), while all remaining parameters are defined in
Remark~\ref{rmk:para}.
\end{remark}

A similar approximation result with \(\operatorname{ReLU}\) activation was also given in~\cite{LuLuWang21} with the following differences:
\begin{enumerate}[(1)]
    \item  Our Barron norm is weaker and gives better dimension dependence than that in~\cite{LuLuWang21}.  
Their Barron norm is defined in a way similar to our \(c\)-Barron norm, 
but it is based on the \(\ell^1\)-norm 
\(\|k\|_1 \coloneqq |k_1| + \cdots + |k_d|\) rather than the Euclidean norm \(|k|\) used here;
\item We provide the explicit upper bounds on the network parameters given in~\eqref{eq:coffebd};
    \item We consider general second-order elliptic PDEs with both Dirichlet and Neumann boundary conditions, 
rather than only the Poisson equation and the stationary Schr\"odinger equation with Neumann boundary conditions.
\end{enumerate}

\section{Proof of the Main Results}
\label{sec:proofsofmain}
The objective of this section is to prove Theorems~\ref{mainthm} and~\ref{mainthm2}, 
following the three-step outline presented in Section~\ref{ourcontri}.
Our argument builds upon the methodology developed in~\cite{Chen21, LuLuWang21, Mar23, CLLZ24}.

\subsection{Step 1: Sobolev Gradient Descent}
\subsubsection{Iterative Schemes}
It is well known that the solutions to the Dirichlet problem~\eqref{Diri} and the Neumann problem~\eqref{Newu} in the Sobolev space setting can be characterized as minimizers of suitable energy functionals. This classical characterization is summarized in the following proposition. 
\begin{proposition}[Variational Characterization of Weak Solutions]
\label{energy}
\hfill
\begin{enumerate}[(1)]
    \item Suppose Assumption~\ref{A1} holds. Let \( u_D^* \in H_0^1(\Omega) \) be the weak solution to the Dirichlet problem~\eqref{Diri}. Then \( u_D^* \) is the unique minimizer of the energy functional
    \begin{equation}
        \label{eq:energyD}
        u_D^* = \mathop{\operatorname{arg\,min}}_{u \in H_0^1(\Omega)} \mathcal{E}_D(u),\quad \mathcal{E}_D(u) \coloneqq \int_{\Omega} \left( \frac{1}{2} \nabla u \cdot A \nabla u + \frac{1}{2} c u^2 - f u  \right) \mathrm{d}x.
    \end{equation}
    
    \item Suppose Assumption~\ref{A1} holds. Let \( u_N^* \in H^1(\Omega) \) be the weak solution to the Neumann problem~\eqref{Newu}. Then \( u_N^* \) is the unique minimizer of the energy functional
    \begin{equation}
       \label{eq:energyN}
        u_N^* = \mathop{\operatorname{arg\,min}}_{u \in H^1(\Omega)} \mathcal{E}_N(u),\quad
        \mathcal{E}_N(u) \coloneqq \int_{\Omega} \left( \frac{1}{2} \nabla u \cdot A \nabla u + \frac{1}{2} c u^2 - f u  \right) \mathrm{d}x.
    \end{equation}
\end{enumerate}
\end{proposition}
\begin{proof}
	We only prove (1), as the proof of (2) follows similarly. By Stampacchia's theorem (see~\cite[Theorem 5.6]{Brezis11}), it suffices to verify that
\[
\int_{\Omega} \nabla u_D^* \cdot A \nabla(v - u_D^*)+cu_D^*(v-u_D^*) \,\mathrm{d}x 
\geq \int_{\Omega} f(v - u_D^*) \,\mathrm{d}x \quad \text{for all } v \in H_0^1(\Omega).
\]
This inequality follows directly from the definition of the weak solution given in~\eqref{uDprop}.
\end{proof}

The classical gradient descent method iteratively updates an initial guess by moving in the direction of the negative gradient of the objective function, aiming to converge to a minimizer. In the context of the PDEs we consider, Proposition~\ref{energy} asserts that the weak solution is the unique minimizer of the associated energy functional. This naturally motivates the application of a gradient descent-type approach to solve the PDE. For the Dirichlet problem, we define the gradient of \( \mathcal{E}_D \) as the map
\[
\mathcal{G}_D \colon H_0^1(\Omega) \longrightarrow H_0^1(\Omega)
\]
such that
\begin{equation}
\label{DH01}
        \big(\mathcal{G}_D(u), v\big)_{H^1(\Omega)}
 = \lim_{t\to 0} \frac{\mathcal{E}_D(u+tv) - \mathcal{E}_D(u)}{t} 
 = \int_{\Omega} \nabla u \cdot A \nabla v + c u v - f v \,\mathrm{d}x,
\end{equation}
for all $u, v \in H_0^1(\Omega)$.
Here we use the inner product on \( H_0^1(\Omega) \) inherited from the standard inner product on \( H^1(\Omega) \). The uniqueness of \( \mathcal{G}_D \) follows from the Riesz representation theorem, and we call \( \mathcal{G}_D \) the \emph{Sobolev gradient} of \( \mathcal{E}_D \) in \( H_0^1(\Omega) \). A similar argument applies to the energy functional \( \mathcal{E}_N \) for the Neumann problem, and there exists a unique operator \( \mathcal{G}_N \colon H^1(\Omega) \to H^1(\Omega) \) such that 
    \begin{equation}
    	\label{NH01}
    \big( \mathcal{G}_N(u), v \big)_{H^1(\Omega)} 
    = \int_{\Omega} \nabla u \cdot A \nabla v + c u v - f v \,\mathrm{d}x,
    \quad \text{for all } u, v \in H^1(\Omega).
    \end{equation}
    We call \( \mathcal{G}_N \) the \emph{Sobolev gradient} of \( \mathcal{E}_N \) in \( H^1(\Omega) \).
We remark that the concept of the Sobolev gradient has been widely employed in the study of PDEs, including applications to the Gross--Pitaevskii eigenvalue problem~\cite{HenP20,Zhang22,CLLZ24}.

Analogous to the fact that the gradient of a function vanishes at its minimizer, the Sobolev gradient of an energy functional should also vanish at its minimizer, i.e.,
\begin{equation}
\label{DE=0}
\mathcal{G}_D(u^*_D) = \mathcal{G}_N(u^*_N) = 0.
\end{equation}
In parallel with the classical gradient descent method,  we define iterative schemes based on the Sobolev gradients  \( \mathcal{G}_D \) and \( \mathcal{G}_N \)  to approximate the solutions of the Dirichlet problem~\eqref{Diri}  and the Neumann problem~\eqref{Newu}, respectively.  These schemes are referred to as \emph{Sobolev gradient descents},  as they can be interpreted as discrete-time approximations  
of the corresponding continuous Sobolev gradient descents  of the energy functionals.

\begin{itemize}
	\item \textbf{Sobolev gradient descent for the Dirichlet problem:}
	\begin{equation}
	\label{Diterationscheme}
	u_{t+1} = u_t - \alpha \mathcal{G}_D(u_t),
	\end{equation}
	with initial condition \( u_0 = 0 \), where \( \alpha > 0 \) is a step size to be specified later.
	
	\item \textbf{Sobolev gradient descent for the Neumann problem:}
	\begin{equation}
	\label{Niterationscheme}
	u_{t+1} = u_t - \alpha \mathcal{G}_N(u_t),
	\end{equation}
	with initial condition \( u_0 = 0 \), where \( \alpha > 0 \) is a step size to be specified later.
\end{itemize}

\subsubsection{Convergence Analysis} This subsection investigates the convergence rate of the Sobolev gradient descents to $u_D^*$ and $u_N^*$. In particular, we have the following theorem.

\begin{theorem}
\label{thm:uttou*}
Assume Assumption~\ref{A1} holds and let $\alpha>0$ be sufficiently small.
The Sobolev gradient descent~\eqref{Diterationscheme} for the Dirichlet problem~\eqref{Diri} converges exponentially in the \( H^1(\Omega) \)-norm to the weak solution \( u_D^* \). More precisely, there exists a constant \( 0 \leq \beta < 1 \) such that
\[
\| u_t - u_D^* \|_{H^1(\Omega)} \leq \| u_D^* \|_{H^1(\Omega)}\beta^t  \quad \text{for all } t \geq 0.
\]
An analogous result holds for the Neumann problem~\eqref{Newu} under the Sobolev gradient descent~\eqref{Niterationscheme}, using the same choice of step size \( \alpha \) and contraction factor \( \beta \).
\end{theorem}

\begin{proof}
We only present the proof for the Dirichlet problem~\eqref{Diri}, as the proof for the Neumann problem~\eqref{Newu} follows analogously.  
Our goal is to establish that there exists a constant \( 0 \leq \beta < 1 \) such that
\[
\| u_{t+1} - u_D^* \|_{H^1(\Omega)} \leq \beta \| u_t - u_D^* \|_{H^1(\Omega)} \quad \text{for all } t \geq 0.
\]

By the construction of the Sobolev gradient descent~\eqref{Diterationscheme}, we compute
\begin{equation}\label{ut+1-u_D*}
    \begin{split}
         \|u_{t+1} - u_D^*\|_{H^1(\Omega)}^2  
=& \|u_t - u_D^* - \alpha \mathcal{G}_D(u_t)\|_{H^1(\Omega)}^2 \\
 =& \|u_t - u_D^*\|_{H^1(\Omega)}^2 
- 2\alpha \left( u_t-u_D^*, \mathcal{G}_D(u_t)\right)_{H^1(\Omega)} 
+ \alpha^2 \| \mathcal{G}_D(u_t) \|_{H^1(\Omega)}^2.
    \end{split}
\end{equation}

\medskip
\noindent \textbf{Estimating \(\bigl(u_t-u_D^*,\mathcal{G}_D(u_t)\bigr)_{H^1(\Omega)}\).}
Let \(e_t\coloneqq u_t-u_D^*\). By~\eqref{DE=0}, we have
\(\mathcal{G}_D(u_D^*)=0\). Hence, using~\eqref{DH01},
\begin{align*}
\bigl(e_t,\mathcal{G}_D(u_t)\bigr)_{H^1(\Omega)}
&=
\bigl(
e_t,
\mathcal{G}_D(u_t)
-
\mathcal{G}_D(u_D^*)
\bigr)_{H^1(\Omega)}
\\
&=
\int_\Omega
\nabla(u_t-u_D^*)\cdot A\nabla e_t
+
c(u_t-u_D^*)e_t\dd x
\\
&=
\int_\Omega
\nabla e_t\cdot A\nabla e_t
+
c|e_t|^2\dd x.
\end{align*}
Letting
$
\lambda_{\min}\coloneqq\min\{a_{\min},c_{\min}\}>0
$,
we obtain
\begin{equation}
\label{eq:u*-ut2G}
\bigl(u_t-u_D^*,\mathcal{G}_D(u_t)\bigr)_{H^1(\Omega)}
\geq
\lambda_{\min}
\int_\Omega
|\nabla e_t|^2+|e_t|^2\dd x
=
\lambda_{\min}
\|u_t-u_D^*\|_{H^1(\Omega)}^2.
\end{equation}

\medskip
\noindent\textbf{Estimating \( \| \mathcal{G}_D(u_t) \|_{H^1(\Omega)}^2 \).}
By~\eqref{DE=0}, we have
\begin{equation}
\label{Est2}
\left( \mathcal{G}_D(u_D^*), \mathcal{G}_D(u_t) \right)_{H^1(\Omega)} = 0 \quad \text{for all } t \geq 0.
\end{equation}
Let $\lambda_{\max} \coloneqq \max\{ a_{\max}, c_{\max} \}$.
Then, by combining~\eqref{Est2}, the definition of the Sobolev gradient in~\eqref{DH01}, and the Cauchy--Schwarz inequality, we obtain
\begin{align*}
\| \mathcal{G}_D(u_t) \|_{H^1(\Omega)}^2 
=& \left( \mathcal{G}_D(u_t), \mathcal{G}_D(u_t) \right)_{H^1(\Omega)} 
- \left( \mathcal{G}_D(u_D^*), \mathcal{G}_D(u_t) \right)_{H^1(\Omega)} \\
=& \int_{\Omega} \nabla (u_t - u_D^*) \cdot A \nabla \mathcal{G}_D(u_t) 
+ c (u_t - u_D^*) \mathcal{G}_D(u_t)  \dd x \\
\leq&\ a_{\max} \| \nabla e_t \|_{L^2(\Omega)} \| \nabla \mathcal{G}_D(u_t) \|_{L^2(\Omega)} 
+ c_{\max} \| e_t \|_{L^2(\Omega)} \| \mathcal{G}_D(u_t) \|_{L^2(\Omega)} \\
\leq&\ \lambda_{\max} \Big( \| \nabla e_t \|_{L^2(\Omega)} \| \nabla \mathcal{G}_D(u_t) \|_{L^2(\Omega)} 
  + \| e_t \|_{L^2(\Omega)} \| \mathcal{G}_D(u_t) \|_{L^2(\Omega)} \Big) \\ 
\leq&\ \lambda_{\max} \left( \| \nabla e_t \|_{L^2(\Omega)}^2 
+ \| e_t \|_{L^2(\Omega)}^2 \right)^{1/2}  \left( \| \nabla \mathcal{G}_D(u_t) \|_{L^2(\Omega)}^2 
+ \| \mathcal{G}_D(u_t) \|_{L^2(\Omega)}^2 \right)^{1/2} \\
=&\ \lambda_{\max} \| e_t \|_{H^1(\Omega)} \| \mathcal{G}_D(u_t) \|_{H^1(\Omega)}.
\end{align*}
This gives
\begin{equation}
\label{eq:Robin u*-ut2}
\| \mathcal{G}_D(u_t) \|_{H^1(\Omega)} 
\leq \lambda_{\max} \| u_t - u_D^* \|_{H^1(\Omega)}.
\end{equation}

Combining the estimates~\eqref{ut+1-u_D*}, \eqref{eq:u*-ut2G}, and~\eqref{eq:Robin u*-ut2}, we obtain
\[
\| u_{t+1} - u_D^* \|_{H^1(\Omega)} 
\leq \left(1 - 2\lambda_{\min} \alpha + \lambda_{\max}^2 \alpha^2\right)^{1/2} \| u_t - u_D^* \|_{H^1(\Omega)}.
\]
Define the contraction factor
\begin{equation}
\label{betadef}
    \beta \colonequals \left(1 - 2\lambda_{\min} \alpha + \lambda_{\max}^2 \alpha^2\right)^{1/2}.
\end{equation}
Then for sufficiently small \( \alpha > 0 \), we can ensure that \( 0 \leq \beta < 1 \), which completes the proof.
\end{proof}

In the subsequent analysis, we use the step size \(\alpha_*\) and the
corresponding contraction factor \(\beta_*\), defined by
\begin{equation}
\label{eq:optimalalpha}
\alpha_{\ast} \coloneqq \frac{ \lambda_{\min} }{  \lambda_{\max}^2 }, \qquad
\beta_{\ast} \coloneqq \left( 1 - \frac{ \lambda_{\min}^2 }{ \lambda_{\max}^2 } \right)^{1/2}.
\end{equation}
This choice of \(\alpha_*\) minimizes the upper bound on the contraction
factor in~\eqref{betadef}.
Moreover, testing the weak formulation~\eqref{uDprop} with
\(v=u_D^*\) and using coercivity and duality, we obtain
\begin{equation}
\label{lfu}
\begin{split}
    \lambda_{\min} \| u_D^* \|_{H^1(\Omega)}^2
& \leq \int_{\Omega} \nabla u_D^* \cdot A \nabla u_D^* \, \mathrm{d}x + \int_{\Omega} c |u_D^*|^2 \, \mathrm{d}x 
\\
& = \int_{\Omega} f u_D^* \, \mathrm{d}x 
\leq \| f \|_{H^{-1}(\Omega)} \| u_D^* \|_{H^1(\Omega)}.
\end{split}
\end{equation}
Consequently,
\begin{equation}
\label{eq:uD*1/lamudamin}
\| u_D^* \|_{H^1(\Omega)} 
\leq \frac{\| f \|_{H^{-1}(\Omega)}}{\lambda_{\min}}.
\end{equation}
Combining Theorem~\ref{thm:uttou*} with the bound in~\eqref{eq:uD*1/lamudamin}, we obtain
\begin{equation}
	\label{setp1D}
	\| u_t - u_D^* \|_{H^1(\Omega)} \leq  \frac{\| f \|_{H^{-1}(\Omega)}}{\lambda_{\min}}\beta_{\ast}^t \quad \text{for all } t \geq 0.
\end{equation}

An analogous estimate holds for the weak solution \( u_N^* \) to the Neumann problem under the Sobolev gradient descent~\eqref{Niterationscheme}:
\begin{equation}
	\label{setp1N}
	\| u_t - u_N^* \|_{H^1(\Omega)} \leq  \frac{\| f \|_{(H^1(\Omega))^*}}{\lambda_{\min}}\beta_{\ast}^t \quad \text{for all } t \geq 0.
\end{equation}
Estimates~\eqref{setp1D} and~\eqref{setp1N} yield the following corollary.
\begin{corollary}
\label{cor:uTepsilon}
Assume that Assumption~\ref{A1} holds, and let \(\alpha_\ast\) and
\(\beta_\ast\) be defined by~\eqref{eq:optimalalpha}. For any
\(\varepsilon>0\), if \(\beta_\ast>0\), choose \(T_1\) and \(T_2\) such that
\begin{equation*}
T_1\geq
\left\lceil
\frac{
\log_+\!\left(
\|f\|_{H^{-1}(\Omega)}/(\varepsilon\lambda_{\min})
\right)}
{|\log\beta_\ast|}
\right\rceil,
\qquad
T_2\geq
\left\lceil
\frac{
\log_+\!\left(
\|f\|_{(H^1(\Omega))^*}/(\varepsilon\lambda_{\min})
\right)}
{|\log\beta_\ast|}
\right\rceil.
\end{equation*}
When \(\beta_\ast=0\), set \(T_1=T_2=1\). Then the iterates \(u_{T_1}\)
and \(u_{T_2}\), generated respectively by
schemes~\eqref{Diterationscheme} and~\eqref{Niterationscheme} with step
size \(\alpha_\ast\), satisfy
\[
\|u_{T_1}-u_D^*\|_{H^1(\Omega)}\leq\varepsilon,
\qquad
\|u_{T_2}-u_N^*\|_{H^1(\Omega)}\leq\varepsilon.
\]
\end{corollary}

\begin{remark}
In the estimate~\eqref{lfu}, we use the \( H^{-1}(\Omega) \)-norm of \( f \) instead of the \( L^2(\Omega) \)-norm, 
since the \( H^{-1} \)-norm is smaller. This follows from the fact that the Poincar\'{e} constant \( C_P \) for the unit hypercube \( \Omega \) is \( 1/(\pi \sqrt{d}) \) (see~Appendix~\ref{sec:poincare}). That is,
\[
\|\phi\|_{L^2(\Omega)} \leq \frac{1}{\pi \sqrt{d}} \|\nabla \phi\|_{L^2(\Omega)}
\quad \text{for all } \phi \in H_0^1(\Omega).
\]
Consequently, for every \( f \in L^2(\Omega) \),
\[
\|f\|_{H^{-1}(\Omega)}
\coloneqq 
\sup\left\{\int_{\Omega} f \phi \,\dd x: \phi\in H_0^1(\Omega),\; \|\nabla \phi\|_{L^2(\Omega)}\leq 1\right\}
\leq \frac{1}{\pi \sqrt{d}} \|f\|_{L^2(\Omega)}.
\]
\end{remark}

\subsection{Step 2: Barron Norm Estimates}
In this section, we focus on proving the following theorem.
\begin{theorem}
\label{thm:baroonnormut+1}
Let \( \alpha_\ast \) be defined as in~\eqref{eq:optimalalpha}, and set
\begin{equation}
	\label{eq:p(d)}
	p(d) \colonequals  \frac{1+\pi}{\pi} \alpha_\ast \ell_{A} d^2 + \alpha_\ast \ell_c + 1.
\end{equation}
\begin{enumerate}[(1)]
	\item Under Assumptions~\ref{A1} and~\ref{A2}, the iterative scheme~\eqref{Diterationscheme} for the Dirichlet problem~\eqref{Diri} satisfies
	\[
	\| u_{t+1} \|_{\mathcal{B}_s^2(\Omega)} 
	\leq p(d) \, \| u_t \|_{\mathcal{B}_s^2(\Omega)} 
	+ \frac{1}{2} \alpha_\ast \ell_{f,D}.
	\]
	
	\item Under Assumptions~\ref{A1} and~\ref{A2'}, the iterative scheme~\eqref{Niterationscheme} for the Neumann problem~\eqref{Newu} satisfies
	\[
	\| u_{t+1} \|_{\mathcal{B}_c^2(\Omega)} 
	\leq p(d) \, \| u_t \|_{\mathcal{B}_c^2(\Omega)} 
	+ \frac{1}{2} \alpha_\ast \ell_{f,N}.
	\]
\end{enumerate}
\end{theorem}

\subsubsection{Resolvent Formulation}
For later use, we define the inverse elliptic operator \( (I - \Delta)^{-1} \) under both Dirichlet and Neumann boundary conditions. Given any \( \varphi \in L^2(\Omega) \), we define:
\begin{itemize}
    \item \( (I - \Delta)_D^{-1} \varphi \in H_0^1(\Omega) \) denotes the unique weak solution to
    \[
    (I - \Delta) w = \varphi \quad \text{in } \Omega, \quad
    w = 0 \quad \text{on } \partial \Omega,
    \]
    \item \( (I - \Delta)_N^{-1} \varphi \in H^1(\Omega) \) denotes the unique weak solution to
    \[
    (I - \Delta) w = \varphi \quad \text{in } \Omega, \quad
    \frac{\partial w}{\partial \nu} = 0 \quad \text{on } \partial \Omega.
    \]
\end{itemize}
The existence and uniqueness of the weak solution in both cases are guaranteed by the Lax--Milgram theorem.

According to~\eqref{DH01}, we have
\[
\int_{\Omega} \mathcal{G}_D(u) \, v + \nabla \mathcal{G}_D(u) \cdot \nabla v \, \mathrm{d}x
= \int_{\Omega} \nabla u \cdot A \nabla v + c u v - f v \, \mathrm{d}x, \quad \text{for all }v \in H_0^1(\Omega).
\]
Thus, \( \mathcal{G}_D(u) \in H_0^1(\Omega) \) is the unique weak solution to the elliptic problem
\[
(I - \Delta) w = \mathcal{L} u - f \quad \text{in } \Omega, \quad
w = 0 \quad  \text{on } \partial \Omega,
\]
where
	\[
	\mathcal{L} u \coloneqq -\nabla \cdot (A \nabla u) + c u
	\]
is the second-order elliptic operator.
This implies that the energy gradient for the Dirichlet problem can be expressed as
\[
\mathcal{G}_D(u) = (I - \Delta)_D^{-1}(\mathcal{L} u - f).
\]
Hence, the iterative scheme~\eqref{Diterationscheme} for the Dirichlet problem~\eqref{Diri} becomes
\begin{equation}
\label{-1DE}
u_{t+1} = u_t - \alpha_{\ast} (I - \Delta)^{-1}_D (\mathcal{L} u_t - f).
\end{equation}

Similarly, along the Neumann iteration, the variational
identity~\eqref{NH01} gives
\[
\mathcal{G}_N(u_t)
=
(I-\Delta)_N^{-1}(\mathcal{L}u_t-f).
\]
Consequently, the iterative scheme~\eqref{Niterationscheme} becomes
\begin{equation}
\label{-1NE}
u_{t+1}
=
u_t-\alpha_\ast
(I-\Delta)_N^{-1}(\mathcal{L}u_t-f).
\end{equation}

\subsubsection{Barron Norm Growth}
A key ingredient in our analysis is the spectral property of the inverse elliptic operator.  
The result stated below will be used repeatedly in this section, and its proof follows from a direct computation.
\begin{proposition}
\label{prop:inverse}
For each \( k \in \mathbb{N}^d \), we have
\[
(I - \Delta)_D^{-1} S_k(x) = \frac{1}{1 + \pi^2 |k|^2} S_k(x), \quad
(I - \Delta)_N^{-1} C_k(x) = \frac{1}{1 + \pi^2 |k|^2} C_k(x).
\]
\end{proposition}

We next collect several algebraic and functional properties of Barron spaces that will be used in the proof of Theorem~\ref{thm:baroonnormut+1}.

\begin{proposition}
\label{prop:linearofBfunc}
Let $\widetilde{\mathcal B}^n(\Omega)$ be any one of the spaces
$\mathcal B_e^n(\Omega)$, $\mathcal B_s^n(\Omega)$,
$\mathcal B_c^n(\Omega)$, and $\mathcal B_{i,j}^n(\Omega)$.
Then the following properties hold.
\begin{enumerate}[(1)]
    \item Linear combination estimate: For all \( g_1, g_2 \in \widetilde{\mathcal{B}}^n(\Omega) \), and \( \lambda_1, \lambda_2 \in \mathbb{R} \),
    \[
    \|\lambda_1 g_1 + \lambda_2 g_2\|_{\widetilde{\mathcal{B}}^n(\Omega)} 
    \leq |\lambda_1| \|g_1\|_{\widetilde{\mathcal{B}}^n(\Omega)} + |\lambda_2| \|g_2\|_{\widetilde{\mathcal{B}}^n(\Omega)}.
    \]

    \item Product estimates:
    \begin{align*}
    \|gh\|_{\mathcal{B}_s^n(\Omega)} 
    &\leq \|g\|_{\mathcal{B}_s^n(\Omega)} \|h\|_{\mathcal{B}_c^n(\Omega)} 
    \quad \text{for } g \in \mathcal{B}_s^n(\Omega), \ h \in \mathcal{B}_c^n(\Omega), \\
    \|gh\|_{\mathcal{B}_c^n(\Omega)} 
    &\leq \|g\|_{\mathcal{B}_c^n(\Omega)} \|h\|_{\mathcal{B}_c^n(\Omega)} 
    \quad \text{for } g, h \in \mathcal{B}_c^n(\Omega).
    \end{align*}

    \item Inverse operator estimates:
    \begin{align*}
    \left\| (I - \Delta)_D^{-1} g \right\|_{\mathcal{B}_s^2(\Omega)} 
    &= \frac{1}{2} \|g\|_{\mathcal{B}_s^0(\Omega)}  
    \quad \text{for } g \in \mathcal{B}_s^0(\Omega), \\
    \left\| (I - \Delta)_N^{-1} g \right\|_{\mathcal{B}_c^2(\Omega)} 
    &= \frac{1}{2} \|g\|_{\mathcal{B}_c^0(\Omega)}  
    \quad \text{for } g \in \mathcal{B}_c^0(\Omega).
    \end{align*}
\end{enumerate}
\end{proposition}

\begin{proof}
Statement (1) follows directly from the definitions.

\medskip
\noindent{(2).}
We prove only the first inequality, as the second one is analogous.  
By the standard uniform convergence argument in analysis,  
the interchange of summation when taking products in the following proofs  
is justified, and the details will be omitted.

Let
\[
g(x) = \sum_{k \in \mathbb{N}^d} a_{k}  S_k(x), \quad
h(x) = \sum_{k' \in \mathbb{N}^d} b_{k'}  C_{k'}(x).
\]
We compute
\begin{align*}
g(x)h(x)
&= \sum_{k, k' \in \mathbb{N}^d} a_k b_{k'} \prod_{j=1}^d \sin\left(\pi k_j x_j\right) \cos\left(\pi k_j' x_j\right) \\
&= \sum_{k, k' \in \mathbb{N}^d} a_k b_{k'} \cdot \frac{1}{2^d} \prod_{j=1}^d \left( \sin\left(\pi(k_j + k_j')x_j\right) + \sin\left(\pi(k_j - k_j')x_j\right) \right) \\
&
=\sum_{k, k' \in \mathbb{N}^d}\sum_{\epsilon\in\{\pm 1\}^d}   \frac{1}{2^d} a_k b_{k'} S_{k+\epsilon\circ k'}(x).
\end{align*}
We estimate
\[
1 + \pi^n |k+\epsilon\circ k'|^n
\leq 1 + \pi^n \left(|k|+|k'|\right)^n
 \leq \left(1 + \pi^n |k|^n\right)\left(1 + \pi^n |k'|^n\right).
\]
Thus, we obtain
\begin{align*}
\|gh\|_{\mathcal{B}_s^n(\Omega)} 
&= \left\| \sum_{k, k' \in \mathbb{N}^d}\sum_{\epsilon\in\{\pm 1\}^d} \frac{1}{2^d} a_k b_{k'}  S_{k+\epsilon\circ k'}(x) \right\|_{\mathcal{B}_s^n(\Omega)} \\
&\leq   \sum_{k, k' \in \mathbb{N}^d} \sum_{\epsilon\in\{\pm 1\}^d}\frac{1}{2^d}|a_kb_{k'}|   \left(1 + \pi^n |k+\epsilon\circ k'|^n\right) \\
&\leq \sum_{k, k' \in \mathbb{N}^d} |a_k| |b_{k'}|  \left(1 + \pi^n |k|^n\right)\left(1 + \pi^n |k'|^n\right) \\
&= \|g\|_{\mathcal{B}_s^n(\Omega)}  \|h\|_{\mathcal{B}_c^n(\Omega)}.
\end{align*}

\medskip
\noindent{(3).}
We prove only the first identity, as the second follows analogously.  
Some standard arguments imply that $(I - \Delta)_D^{-1}$ and the summation are interchangeable, i.e.,
\[
(I - \Delta)_D^{-1} g = \sum_{k \in \mathbb{N}^d} a_k (I - \Delta)_D^{-1} S_k(x).
\]
It now follows from Proposition~\ref{prop:inverse} and the definition of the \(s\)-Barron norm that
\[
\left\| (I - \Delta)_D^{-1} g \right\|_{\mathcal{B}_s^2(\Omega)} 
= \left\| \sum_{k \in \mathbb{N}^d} a_k \frac{1}{1 + \pi^2 |k|^2}  S_k(x) \right\|_{\mathcal{B}_s^2(\Omega)} = \sum_{k \in \mathbb{N}^d} |a_k|
=\frac{1}{2}\|g\|_{\mathcal{B}_s^0(\Omega)}.
\]
\end{proof}

\begin{lemma}
\label{lem:Ajterms}
Suppose the condition on \( A(x) \) in Assumption~\ref{A2} holds, and let \( u \in \mathcal{B}_s^2(\Omega) \). Then for every \( i, j = 1, \dots, d \), we have
\begin{align}
\left\| (I - \Delta)^{-1}_D (\partial_i A_{ij} \partial_j u) \right\|_{\mathcal{B}_s^2(\Omega)}
&\leq \frac{1}{\pi} \ell_{A} \|u\|_{\mathcal{B}_s^2(\Omega)}, \label{pApu} \\
\left\| (I - \Delta)^{-1}_D (A_{ij} \partial_{ij} u) \right\|_{\mathcal{B}_s^2(\Omega)}
&\leq \ell_{A} \|u\|_{\mathcal{B}_s^2(\Omega)}. \label{Appu}
\end{align}
Under Assumption~\ref{A2'} and for \( u \in \mathcal{B}_c^2(\Omega) \), the same estimates hold with \( (I - \Delta)^{-1}_N \), \( \ell_{A} \), and \( \mathcal{B}_c^2(\Omega) \).
\end{lemma}
\begin{proof}
We prove the case under Assumption~\ref{A2}; the case under Assumption~\ref{A2'} follows by the same argument.
By the standard uniform convergence argument in analysis,  
the interchange of the order of summation and differentiation,  
as well as the interchange of summation when taking products in the following proofs,  
is justified.  
We will omit the details when applying these arguments.

\medskip
\noindent\textbf{Proof of~\eqref{pApu} in the case \( i = j \).}
Let
\[
A_{ii}(x) = \sum_{k \in \mathbb{N}^d} a_k C_k(x), \quad 
u(x) = \sum_{k' \in \mathbb{N}^d} b_{k'} S_{k'}(x).
\]
We compute
\begin{align*}
\partial_i A_{ii} \partial_i u 
=& \sum_{k, k' \in \mathbb{N}^d} a_k b_{k'} \left( \partial_i C_k  \partial_i S_{k'} \right) \\
=&   \sum_{k, k' \in \mathbb{N}^d}-\pi^2 a_k b_{k'} k_i k_i' 
\left( \prod_{\substack{1 \leq m \leq d \\ m \neq i}} \cos\left( \pi k_m x_m \right) \sin\left( \pi k_m' x_m \right) \right)\cdot \sin\left( \pi k_i x_i \right) \cos\left( \pi k_i' x_i \right) \\
=& \sum_{k, k' \in \mathbb{N}^d} -\pi^2a_k b_{k'} k_i k_i' \cdot \frac{1}{2^d} 
\prod_{\substack{1 \leq m \leq d \\ m \neq i}} \left( \sin\left( \pi(k_m + k_m') x_m \right) - \sin\left( \pi(k_m - k_m') x_m \right) \right)   \\
&\qquad \cdot \left( \sin\left( \pi(k_i + k_i') x_i \right) + \sin\left( \pi(k_i - k_i') x_i \right) \right) \\
=&  \sum_{k, k' \in \mathbb{N}^d}\sum_{\epsilon\in\{\pm 1\}^d}  -\frac{\pi^2 }{2^d}  a_k b_{k'} k_i k_i' \epsilon_i\prod_{m=1}^{d}\epsilon_m \cdot S_{k+\epsilon\circ k' }(x),
\end{align*}
Thus, we obtain
\begin{align*}
&\|(I - \Delta)^{-1}_D(\partial_i A_{ii}  \partial_i u)\|_{\mathcal{B}_s^2(\Omega)} \\
=& \left\|  \sum_{k, k' \in \mathbb{N}^d}\sum_{\epsilon\in\{\pm 1\}^d}  -\frac{\pi^2 }{2^d}  a_k b_{k'} k_i k_i' \epsilon_i\prod_{m=1}^{d}\epsilon_m \cdot  (I - \Delta)^{-1}_D S_{k+\epsilon\circ k'}(x) \right\|_{\mathcal{B}_s^2(\Omega)} \\
=& \left\|  \sum_{k, k' \in \mathbb{N}^d}\sum_{\epsilon\in\{\pm 1\}^d}  -\frac{\pi^2 }{2^d}  a_k b_{k'} k_i k_i' \epsilon_i\prod_{m=1}^{d}\epsilon_m \cdot  \frac{1}{1 + \pi^2 |k+\epsilon\circ k'|^2} S_{k+\epsilon\circ k'}(x) \right\|_{\mathcal{B}_s^2(\Omega)} \\
\leq& \sum_{k, k' \in \mathbb{N}^d} \pi^2 |a_kb_{k'}|  k_ik_i' 
\leq  \sum_{k, k' \in \mathbb{N}^d} \pi^2 |a_k| |b_{k'}| |k| |k'|^2 \\
\leq& \frac{1}{\pi} \|A_{ii}\|_{\mathcal{B}_c^1(\Omega)} \|u\|_{\mathcal{B}_s^2(\Omega)} \leq \frac{1}{\pi}\ell_{A} \|u\|_{\mathcal{B}_s^2(\Omega)}.
\end{align*}

\medskip
\noindent\textbf{Proof of~\eqref{pApu} in the case \( i \neq j \).}
Suppose
\[
A_{ij}(x)=\sum_{k \in \mathbb{N}^d} a_k M_{k,i,j}(x), \quad 
u(x) = \sum_{k' \in \mathbb{N}^d} b_{k'} S_{k'}(x).
\]
We compute
\begin{align*}
\partial_i A_{ij}  \partial_j u 
=&  \sum_{k, k' \in \mathbb{N}^d} \pi^2a_k b_{k'} k_i k_j' 
\left( \prod_{\substack{1 \leq m \leq d \\ m \neq i,j}} \cos\left( \pi k_m x_m \right) \sin\left( \pi k_m' x_m \right) \right) \\
&\qquad \cdot  \cos\left( \pi k_i x_i \right) \sin\left( \pi k_i' x_i \right) 
\sin\left( \pi k_j x_j \right) \cos\left( \pi k_j' x_j \right) \\
=& \sum_{k, k' \in \mathbb{N}^d} \pi^2 a_k b_{k'} k_i k_j' \cdot \frac{1}{2^d} 
\prod_{\substack{1 \leq m \leq d \\ m \neq j}} \left( \sin\left( \pi(k_m + k_m') x_m \right) - \sin\left( \pi(k_m - k_m') x_m \right) \right) \\
&\qquad \cdot \left( \sin\left( \pi(k_j + k_j') x_j \right) + \sin\left( \pi(k_j - k_j') x_j \right) \right) \\
=&  \sum_{k, k' \in \mathbb{N}^d}\sum_{\epsilon\in\{\pm 1\}^d} \frac{\pi^2}{2^d}a_k b_{k'} k_i k_j' 
\epsilon_j\prod_{m=1}^{d}\epsilon_m \cdot S_{k+\epsilon\circ k' }(x),
\end{align*}
Thus, we obtain
\begin{align*}
&\|(I - \Delta)^{-1}_D(\partial_{i}A_{ij} \partial_{j} u)\|_{\mathcal{B}_s^2(\Omega)}\\
=& \left\|  \sum_{k, k' \in \mathbb{N}^d}\sum_{\epsilon\in\{\pm 1\}^d} \frac{\pi^2}{2^d}a_k b_{k'} k_i k_j' 
\epsilon_j\prod_{m=1}^{d}\epsilon_m \cdot  (I - \Delta)^{-1}_D S_{k+\epsilon\circ k'}(x) \right\|_{\mathcal{B}_s^2(\Omega)} \\
=& \left\| \sum_{k, k' \in \mathbb{N}^d}\sum_{\epsilon\in\{\pm 1\}^d} \frac{\pi^2}{2^d}a_k b_{k'} k_i k_j' 
\epsilon_j\prod_{m=1}^{d}\epsilon_m \cdot  \frac{1}{1 + \pi^2 |k+\epsilon\circ k'|^2} S_{k+\epsilon\circ k'}(x) \right\|_{\mathcal{B}_s^2(\Omega)} \\
\leq&  \sum_{k, k' \in \mathbb{N}^d} \pi^2 |a_kb_{k'}|  k_i k_j' 
\leq  \sum_{k, k' \in \mathbb{N}^d} \pi^2 |a_k| |b_{k'}| |k| |k'|^2 \\
\leq& \frac{1}{\pi} \|A_{ij}\|_{\mathcal{B}_{i,j}^1(\Omega)} \|u\|_{\mathcal{B}_s^2(\Omega)}\leq \frac{1}{\pi}\ell_{A} \|u\|_{\mathcal{B}_s^2(\Omega)}.
\end{align*}

\medskip
\noindent\textbf{Proof of~\eqref{Appu} in the case \( i = j \).}
Let
\[
A_{ii}(x) = \sum_{k \in \mathbb{N}^d} a_k C_k(x), \quad 
u(x) = \sum_{k' \in \mathbb{N}^d} b_{k'} S_{k'}(x).
\]
We compute
\begin{align*}
A_{ii} \partial_{ii} u 
=&  \sum_{k, k' \in \mathbb{N}^d} a_k b_{k'} \left( C_k \partial_{ii} S_{k'} \right) \\
=&  \sum_{k, k' \in \mathbb{N}^d} -\pi^2a_k b_{k'} k_i'^2 \prod_{m=1}^d \cos\left( \pi k_m x_m \right) \sin\left( \pi k_m' x_m \right) \\
=&  \sum_{k, k' \in \mathbb{N}^d} -\pi^2 a_k b_{k'} k_i'^2 \cdot \frac{1}{2^d} \prod_{m=1}^d \left( \sin\left( \pi(k_m + k_m') x_m \right) - \sin\left( \pi(k_m - k_m') x_m \right) \right) \\
=&  \sum_{k, k' \in \mathbb{N}^d}\sum_{\epsilon\in\{\pm 1\}^d}- \frac{\pi^2}{2^d}a_k b_{k'} k_i'^2  \prod_{m=1}^{d}\epsilon_m \cdot S_{k+\epsilon\circ k' }(x).
\end{align*}
Thus, we obtain
\begin{align*}
&\|(I - \Delta)^{-1}_D(A_{ii} \partial_{ii} u)\|_{\mathcal{B}_s^2(\Omega)} \\
=& \left\| \sum_{k, k' \in \mathbb{N}^d}\sum_{\epsilon\in\{\pm 1\}^d}- \frac{\pi^2}{2^d}a_k b_{k'} k_i'^2  \prod_{m=1}^{d}\epsilon_m \cdot   (I - \Delta)^{-1}_D S_{k+\epsilon\circ k'} \right\|_{\mathcal{B}_s^2(\Omega)} \\
=& \left\| \sum_{k, k' \in \mathbb{N}^d}\sum_{\epsilon\in\{\pm 1\}^d}- \frac{\pi^2}{2^d}a_k b_{k'} k_i'^2  \prod_{m=1}^{d}\epsilon_m \cdot  \frac{1}{1 + \pi^2 |k+\epsilon\circ k'|^2} S_{k+\epsilon\circ k'} \right\|_{\mathcal{B}_s^2(\Omega)} \\
\leq&  \sum_{k, k' \in \mathbb{N}^d} \pi^2 |a_kb_{k'}|  k_i'^2 \leq 
 \sum_{k, k' \in \mathbb{N}^d} \pi^2 |a_k| |b_{k'}| |k'|^2 \\
\leq& \|A_{ii}\|_{\mathcal{B}_c^1(\Omega)} \|u\|_{\mathcal{B}_s^2(\Omega)}\leq \ell_{A} \|u\|_{\mathcal{B}_s^2(\Omega)}.
\end{align*}

\medskip
\noindent\textbf{Proof of~\eqref{Appu} in the case \( i \neq j \).}
Suppose
\[
A_{ij}(x)=\sum_{k \in \mathbb{N}^d} a_k M_{k,i,j}(x), \quad 
u(x) = \sum_{k' \in \mathbb{N}^d} b_{k'} S_{k'}(x).
\]
We compute
\begin{align*}
&A_{ij}  \partial_{ij} u \\
=&   \sum_{k, k' \in \mathbb{N}^d} \pi^2a_k b_{k'} k_i' k_j' 
\left( \prod_{\substack{1 \leq m \leq d \\ m \neq i,j}} \cos\left( \pi k_m x_m \right) \sin\left( \pi k_m' x_m \right) \right) \\
&\qquad \cdot  \sin\left( \pi k_i x_i \right) \cos\left( \pi k_i' x_i \right) 
\sin\left( \pi k_j x_j \right) \cos\left( \pi k_j' x_j \right) \\
=&  \sum_{k, k' \in \mathbb{N}^d}\pi^2 a_k b_{k'} k_i' k_j' \cdot \frac{1}{2^d} 
\prod_{\substack{1 \leq m \leq d \\ m \neq i,j}} \left( \sin\left( \pi(k_m + k_m') x_m \right) - \sin\left( \pi(k_m - k_m') x_m \right) \right) \\
&\qquad \cdot  \left( \sin( \pi(k_i + k_i') x_i ) + \sin( \pi(k_i - k_i') x_i ) \right) \left( \sin( \pi(k_j + k_j') x_j) + \sin( \pi(k_j - k_j') x_j ) \right)   \\
=&   \sum_{k, k' \in \mathbb{N}^d}\sum_{\epsilon\in\{\pm 1\}^d} \frac{\pi^2}{2^d} a_k b_{k'} k_i' k_j'  
\epsilon_i\epsilon_j\prod_{m=1}^{d}\epsilon_m \cdot S_{k+\epsilon\circ k' }(x),
\end{align*}
Thus, we obtain
\begin{align*}
&\|(I - \Delta)^{-1}_D(A_{ij} \partial_{ij} u)\|_{\mathcal{B}_s^2(\Omega)} \\
=& \left\| \sum_{k, k' \in \mathbb{N}^d}\sum_{\epsilon\in\{\pm 1\}^d} \frac{\pi^2}{2^d} a_k b_{k'} k_i' k_j'  
\epsilon_i\epsilon_j\prod_{m=1}^{d}\epsilon_m \cdot (I - \Delta)^{-1}_D S_{k+\epsilon\circ k'} \right\|_{\mathcal{B}_s^2(\Omega)} \\
=& \left\| \sum_{k, k' \in \mathbb{N}^d}\sum_{\epsilon\in\{\pm 1\}^d} \frac{\pi^2}{2^d} a_k b_{k'} k_i' k_j'  
\epsilon_i\epsilon_j\prod_{m=1}^{d}\epsilon_m \cdot \frac{1}{1 + \pi^2 |k+\epsilon\circ k'|^2} S_{k+\epsilon\circ k'} \right\|_{\mathcal{B}_s^2(\Omega)} \\
\leq& \sum_{k, k' \in \mathbb{N}^d} \pi^2 |a_kb_{k'}|  k_i' k_j' \leq 
\sum_{k, k' \in \mathbb{N}^d} \pi^2 |a_k| |b_{k'}| |k'|^2 \\
\leq& \|A_{ij}\|_{\mathcal{B}_{i,j}^1(\Omega)} \|u\|_{\mathcal{B}_s^2(\Omega)}\leq \ell_{A} \|u\|_{\mathcal{B}_s^2(\Omega)}.
\end{align*}
\end{proof}

\begin{lemma}
\label{lem:cut}
\hfill
\begin{enumerate}[(1)]
	\item Suppose the condition on \( c(x) \) in Assumption~\ref{A2} holds, and let \( u \in \mathcal{B}_s^2(\Omega) \).  Then we have
	\[\|(I-\Delta)^{-1}_D(cu)\|_{\CB_s^2(\Omega)}\leq \ell_{c}\|u\|_{\CB_s^2(\Omega)}.\]
	\item Suppose the condition on \( c(x) \) in Assumption~\ref{A2'} holds, and let \( u\in \CB_c^2(\Omega) \). Then we have
	\[\|(I-\Delta)^{-1}_N(cu)\|_{\CB_c^2(\Omega)}\leq \ell_{c}\|u\|_{\CB_c^2(\Omega)}.\]
\end{enumerate}
\end{lemma}
\begin{proof}
We provide the proof of (1), as the proof of (2) follows by the same argument.
By part (2) of Proposition~\ref{prop:linearofBfunc}, we have \( cu \in \mathcal{B}_s^2(\Omega) \).  
It then follows from part (3) of the same proposition that
\[
\|(I - \Delta)^{-1}_D(cu)\|_{\mathcal{B}_s^2(\Omega)} \leq \|cu\|_{\mathcal{B}_s^2(\Omega)}.
\]
Applying part (2) again yields the desired result.
\end{proof}

\begin{proof}[\textbf{Proof of Theorem~\ref{thm:baroonnormut+1}}]
Note that
\[
\mathcal{L} u - f=-\nabla \cdot (A \nabla u) + c u - f
= - \sum_{i, j = 1}^{d} \partial_i A_{ij} \partial_j u
- \sum_{i, j = 1}^{d} A_{ij} \partial_{ij} u
+ c u - f.
\]
The result follows from the linearity of \( (I - \Delta)^{-1}_D \) and \( (I - \Delta)^{-1}_N \), together with~\eqref{-1DE}, \eqref{-1NE}, Proposition~\ref{prop:linearofBfunc}, and Lemmas~\ref{lem:Ajterms} and~\ref{lem:cut}.
\end{proof}

\subsection{Step 3: Neural Network Approximation}
\label{sec3.3}
The following approximation theorem for Barron functions by neural networks is analogous to \cite[Theorem 4]{E}, and our proof strategy is a refinement of their argument. To the best of our knowledge, this remains the standard approach in the literature; similar strategies are employed in the approximation results for Barron functions in~\cite{Chen21, CLLZ23, Feng25}.
\begin{theorem}
\label{thm:nueraltobarron}
Let \(g \in \mathcal{B}_e^2(\Omega)\) be an \(e\)-Barron function.
Then, for any fixed positive integer \(k\), there exists a two-layer neural network of the form
\[
N_k(x)
=
\frac{1}{k}\sum_{i=1}^k
a_i\cos\bigl(w_i^\top x+b_i\bigr),
\]
whose parameters and \(e\)-Barron norm satisfy
\begin{equation}
\label{eq:Nkgbarron}
|a_i|\leq \|g\|_{\mathcal{B}_e^2(\Omega)},\quad
w_i\in \pi\BZ^d,\quad
\frac{1}{k} \sum_{i=1}^k |a_i|\bigl(1+|w_i|^2\bigr)
\leq \|g\|_{\mathcal{B}_e^2(\Omega)},\quad
\|N_k\|_{\mathcal{B}_e^2(\Omega)}
\leq \|g\|_{\mathcal{B}_e^2(\Omega)},
\end{equation}
for \(i=1,\ldots,k\), and such that
\begin{equation}
\label{eq:nerultobarron}
\left\|N_k-g\right\|_{H^1(\Omega)}
\leq
\frac{\|g\|_{\mathcal{B}_e^2(\Omega)}}{\sqrt{k}}.
\end{equation}
\end{theorem}
\begin{proof}
The case \(g=0\) is trivial, so we assume \(g\neq 0\). Fix
\(\varepsilon>0\). By the definition of the \(e\)-Barron norm, there
exists an expansion
\begin{equation}
	\label{gReg}
	g(x) = \sum_{\omega \in \mathbb{Z}^d} c_{\omega} e^{i \pi \omega^{\top} x}
\end{equation}
such that
\begin{equation}
	\label{gepsilon}
\sum_{\omega \in \mathbb{Z}^d} |c_{\omega}| \left(1 + \pi^2 |\omega|^2\right) 
\leq \|g\|_{\mathcal{B}_e^2(\Omega)} + \varepsilon.
\end{equation}
Since \(g\) is real-valued, writing
\(c_\omega=|c_\omega|e^{i\theta_\omega}\) and taking real
parts in~\eqref{gReg} yield
\begin{equation}
\label{g=cw}
g(x)=
\sum_{\omega\in\mathbb{Z}^d}
|c_\omega|
\cos\left(\pi\omega^\top x+\theta_\omega\right).
\end{equation}

By \(g\neq 0\) and~\eqref{gepsilon}, we have
\[
0< Z \coloneqq  \sum_{\omega \in \mathbb{Z}^d} |c_{\omega}| \left(1 + \pi^2 |\omega|^2\right)
\leq \|g\|_{\mathcal{B}_e^2(\Omega)}+\varepsilon
 < \infty.
\]
Thus, the measure
\begin{equation}
	\label{eq:mudef}
\mu(a,w,b) \colonequals \sum_{\omega \in \mathbb{Z}^d} \frac{|c_{\omega}|\left(1 + \pi^2 |\omega|^2\right)}{Z}  \delta_{(Z/(1 + \pi^2 |\omega|^2) , \pi \omega, \theta_{\omega})}
\end{equation}
defines a probability measure on \(\mathbb{R} \times \mathbb{R}^d \times \mathbb{R}\),  
where \(\delta_{(Z/(1 + \pi^2 |\omega|^2) , \pi \omega, \theta_{\omega})}\) denotes the Dirac measure on \(\mathbb{R} \times \mathbb{R}^d \times \mathbb{R}\).
We equip the space \((\mathbb{R} \times \mathbb{R}^d \times \mathbb{R})^k\), endowed with its Borel \(\sigma\)-algebra, with the product probability measure \({\mu(a,w,b)}^{\otimes k}\).
Let
\[
\Theta \colonequals \left\{ (a_i, w_i, b_i) \right\}_{1 \leq i \leq k} \in (\mathbb{R} \times \mathbb{R}^d \times \mathbb{R} )^k,
\]
and define the pointwise approximation error between the neural network and the target function \( g \) by
\[
R(\Theta; x) \colonequals \frac{1}{k} \sum_{i=1}^k a_i  \sigma\left( w_i^{\top} x + b_i \right) - g(x),
\]
which is a random variable on the probability space \(\left((\mathbb{R} \times \mathbb{R}^d \times \mathbb{R})^k, {\mu(a,w,b)}^{\otimes k}\right)\).  Then to prove the theorem, it suffices to establish that
\begin{equation}
\label{eq:ETheta}
\mathbb{E}_{\mu^{\otimes k}} \left\| R(\Theta; x) \right\|^2_{H^1(\Omega)} \leq \frac{\|g\|_{\mathcal{B}_e^2(\Omega)}^2}{k}.
\end{equation}

Observe that the expected \( H^1 \)-error decomposes as
\begin{equation}
    \label{deexp}
\mathbb{E}_{\mu^{\otimes k}} \left\| R(\Theta; x) \right\|^2_{H^1(\Omega)}
=\mathbb{E}_{\mu^{\otimes k}} \left\| R(\Theta; x) \right\|^2_{L^2(\Omega)}+\mathbb{E}_{\mu^{\otimes k}}\| \nabla R(\Theta; x) \|_{L^2(\Omega)}^2.
\end{equation}
We proceed to estimate the two terms on the right-hand side separately.

\medskip
\noindent\textbf{Expectation of $\| R(\Theta; x) \|_{L^2(\Omega)}^2$.}
Since we assume that \( \sigma \) is the cosine function, it follows from~\eqref{eq:mudef} and~\eqref{g=cw} that
\begin{equation}
\label{eq.Expg}
\mathbb{E}_{\mu(a,w,b)} \left[ a \sigma(w^{\top} x + b) \right]=
\int_{\mathbb{R} \times \mathbb{R}^d \times \mathbb{R}} a \cos\left(  w^{\top} x + b \right)  \dd \mu(a, w, b) 
= g(x).  
\end{equation}

It follows from \eqref{eq.Expg} that for every \(1 \leq i \neq j \leq k\), we have
\begin{equation}
	\label{expRij}
\int_{(\mathbb{R} \times \mathbb{R}^d \times \mathbb{R})^k}
\left( a_i \sigma(w_i^{\top} x + b_i) - g(x) \right)
\left( a_j \sigma(w_j^{\top} x + b_j) - g(x) \right)
 \dd {\mu}^{\otimes k} = 0.
\end{equation}
Thus, combining~\eqref{expRij} and~\eqref{eq.Expg}, we compute that
\begin{align*}
 \mathbb{E}_{\mu^{\otimes k}} \left\| R(\Theta; x) \right\|^2_{L^2(\Omega)}
=& \int_{(\mathbb{R} \times \mathbb{R}^d \times \mathbb{R})^k}
\int_{\Omega} \left( \frac{1}{k} \sum_{i=1}^k a_i \sigma(w_i^{\top} x + b_i) - g(x) \right)^2 \dd x  \mathrm{d}{\mu}^{\otimes k} \\
=& \frac{1}{k^2} \int_{\Omega} \int_{(\mathbb{R} \times \mathbb{R}^d \times \mathbb{R})^k}
\left( \sum_{i=1}^k \left(a_i  \sigma(w_i^{\top} x + b_i) -  g(x)\right) \right)^2 \dd {\mu}^{\otimes k}  \mathrm{d}x \\
=& \frac{1}{k^2} \int_{\Omega} \sum_{i=1}^k 
\int_{(\mathbb{R} \times \mathbb{R}^d \times \mathbb{R})^k} 
\left( a_i \sigma(w_i^{\top} x + b_i) - g(x) \right)^2 \dd {\mu}^{\otimes k}  \mathrm{d}x \\
=& \frac{1}{k} \int_{\Omega} \int_{\mathbb{R} \times \mathbb{R}^d \times \mathbb{R}} 
\left( a \sigma(w^{\top} x + b) - g(x) \right)^2 \dd \mu  \mathrm{d}x\\
=& \frac{1}{k} \int_{\Omega} \operatorname{Var}_{\mu(a, w, b)} \left( a \sigma(w^{\top} x + b) \right) \dd x.
\end{align*}
Since \(\sigma\) is the cosine function, we have
\[
\operatorname{Var}_{\mu(a, w, b)} \left( a \sigma(w^{\top} x + b) \right)\leq 
\mathbb{E}_{\mu(a, w, b)} \left[ \left(a \sigma(w^{\top} x + b)\right)^2 \right] 
\leq \mathbb{E}_{\mu(a, w, b)} \left[ a^2 \right].
\]
Thus,
\begin{equation}
\label{eq:expED}
\mathbb{E}_{\mu^{\otimes k}} \left\| R(\Theta; x) \right\|^2_{L^2(\Omega)} 
\leq \frac{m(\Omega)}{k}\mathbb{E}_{\mu(a, w, b)} \left[ a^2 \right]
=\frac{1}{k}\mathbb{E}_{\mu(a, w, b)} \left[ a^2 \right],
\end{equation}
where \( m(\Omega) \) denotes the Lebesgue measure of \( \Omega \), which equals \(1\) since \( \Omega = (0,1)^d \).

\medskip
\noindent\textbf{Expectation of $\| \nabla R(\Theta; x) \|_{L^2(\Omega)}^2$.}
Let \(\langle w_i, e_j \rangle\) denote the standard Euclidean inner product of \(w_i\) and \(e_j\) in \(\mathbb{R}^d\); that is, \(\langle w_i, e_j \rangle\) is the \(j\)-th component of \(w_i\in \mathbb{R}^d\).
Differentiating both sides of~\eqref{eq.Expg} with respect to \( x_j \),  
it is straightforward to verify that for every \(1 \leq i \neq \ell \leq k\) and for each \(1 \leq j \leq d\), we have
\begin{align*}
&\int_{(\mathbb{R} \times \mathbb{R}^d \times \mathbb{R})^k}
\left( a_i\langle w_i,e_j\rangle \sigma'(w_i^{\top} x + b_i) - \partial_j g(x) \right)
\\
& \qquad\qquad\qquad\qquad\cdot\left( a_{\ell}\langle w_{\ell},e_j\rangle \sigma'(w_{\ell}^{\top} x + b_{\ell}) - \partial_j g(x) \right)
 \dd \mu^{\otimes k} = 0,
\end{align*}
and
\begin{equation*}
    \mathbb{E}_{(a, w, b)} \left[ a \left\langle w, e_j\right\rangle\sigma'(w^{\top} x + b) \right]=\partial_j g(x).
\end{equation*}
Thus, for every \(1 \leq j \leq d\), we have
\begin{align*}
	&\mathbb{E}_{\mu^{\otimes k}} \left\| \partial_j R(\Theta; x) \right\|^2_{L^2(\Omega)}  \\
	=&\int_{\left(\mathbb{R} \times \mathbb{R}^d \times \mathbb{R}\right)^k} \int_{\Omega}\left(\partial_j\left(\frac{1}{k} \sum_{i=1}^k a_i \sigma\left(w_i^{\top} x+b_i\right)\right)-\partial_j g(x)\right)^2 \dd x \mathrm{d} \mu^{\otimes k}\\
	=&\frac{1}{k^2} \int_{\Omega}\sum_{i=1}^k \int_{\left(\mathbb{R} \times \mathbb{R}^d \times \mathbb{R}\right)^k} \left(a_i\left\langle w_i, e_j\right\rangle \sigma^{\prime}\left(w_i^{\top} x+b_i\right)-\partial_j g(x)\right)^2 \dd  \mu^{\otimes k} \mathrm{d}x\\
	=&\frac{1}{k} \int_{\Omega} \int_{\mathbb{R} \times \mathbb{R}^d \times \mathbb{R}}\left(a\left\langle w, e_j\right\rangle \sigma^{\prime}\left(w^{\top} x+b\right)-\partial_j g(x)\right)^2 \dd \mu(a, w, b) \mathrm{d}x\\
	=&\frac{1}{k} \int_{\Omega} \operatorname{Var}_{\mu(a, w, b)} \left(a\left\langle w, e_j\right\rangle \sigma^{\prime}\left(w^{\top} x+b\right)\right) \dd x\\
	\leq& \frac{1}{k} \int_{\Omega}\mathbb{E}_{\mu(a, w, b)} \left[\left( a\left\langle w, e_j\right\rangle \sigma^{\prime}\left(w^{\top} x+b\right)\right)^2 \right] \dd x.
\end{align*}
Since \(\sigma\) is the cosine function, we have	
\begin{equation}
\label{eq:grad-L2-bound}
\begin{split}
    & \mathbb{E}_{\mu^{\otimes k}}\| \nabla R(\Theta; x) \|_{L^2(\Omega)}^2=
	\mathbb{E}_{\mu^{\otimes k}}\sum_{j=1}^{d} \left\| \partial_j R(\Theta; x) \right\|^2_{L^2(\Omega)}  \\
	\leq& \frac{1}{k} \int_{\Omega}\sum_{j=1}^{d}\mathbb{E}_{\mu(a, w, b)} \left[\left( a\left\langle w, e_j\right\rangle \sigma^{\prime}\left(w^{\top} x+b\right)\right)^2 \right] \dd x  \\
	\leq& \frac{1}{k} \int_{\Omega}\sum_{j=1}^{d}\mathbb{E}_{\mu(a, w, b)} \left[\left( a\left\langle w, e_j\right\rangle \right)^2\right] \dd x  \\
	=&  \frac{1}{k}\mathbb{E}_{\mu(a, w, b)} \left[ a^2|w|^2 \right].
\end{split}
\end{equation}

Combining~\eqref{deexp}, \eqref{eq:expED}, and~\eqref{eq:grad-L2-bound}, we obtain
\[
\mathbb{E}_{\mu^{\otimes k}} \left\| R(\Theta; x) \right\|^2_{H^1(\Omega)}
\leq \frac{1}{k} \mathbb{E}_{\mu(a, w, b)}\left[a^2(1+|w|^2) \right] 
= \frac{Z}{k} \left( \sum_{\omega \in \mathbb{Z}^d} |c_{\omega}| \right) 
\leq \frac{\left(\|g\|_{\mathcal{B}_e^2(\Omega)}+\varepsilon\right)^2}{k}.
\]
Letting \(\varepsilon \to 0\) yields the desired estimate~\eqref{eq:ETheta}.
Moreover, estimate~\eqref{eq:Nkgbarron} follows from the representation~\eqref{g=cw} in terms of exponential functions and the definition~\eqref{eq:mudef}.
This completes the proof of the theorem.
\end{proof}

\begin{proof}[\textbf{Proof of Theorem~\ref{mainthm}}]
The conclusion follows by combining Proposition~\ref{prop:barron},
Corollary~\ref{cor:uTepsilon}, Theorem~\ref{thm:baroonnormut+1}, Proposition~\ref{prop:linearofBfunc}, and Theorem~\ref{thm:nueraltobarron}.
\end{proof}

In order to prove the approximation result for the \(\operatorname{ReLU}\) activation,
we first state Theorem~\ref{thmllw21}, a modified version of~\cite[Theorem~17]{LuLuWang21}
adapted to our definition of the Barron norm.
The Barron norm in~\cite{LuLuWang21} is defined similarly to our \(c\)-Barron norm,
except that it is based on the \(\ell^1\)-norm \(\|k\|_1\) rather than the Euclidean norm
\(|k|\).
Consequently, our norm is weaker and gives better dimension dependence.
Nevertheless, the proofs of Lemma~18 and Proposition~19 in~\cite{LuLuWang21}, which are
the key ingredients in the proof of Theorem~17 therein, can be adapted to our setting
with minor modifications. For completeness, we provide the full proof in
Appendix~\ref{genllw21}.

\begin{theorem}
\label{thmllw21}
Suppose \(g \in \CB_s^2(\Omega)\) is an \( s \)-Barron function of weight 2. Then, for any fixed positive integer \( k \), there exist a constant \( \gamma \in \mathbb{R} \) with \( |\gamma| \leq 2\|g\|_{\CB_s^2(\Omega)} \), and a collection of parameters
\[
\left\{ (a_i, w_i, b_i) \in \mathbb{R} \times \mathbb{R}^d \times \mathbb{R} \right\}_{1 \leq i \leq k}
\]
satisfying
\[
\sum_{i=1}^k |a_i| \leq 4\sqrt{d}\,\|g\|_{\CB_s^2(\Omega)}, 
\quad |w_i| = 1, \quad |b_i| \leq \sqrt{d}, \quad \text{for all } 1 \leq i \leq k,
\]
such that
\[
\left\|
\gamma+\sum_{i=1}^{k} a_i \operatorname{ReLU}(w_i^{\top} x + b_i) - g(x)
\right\|_{H^1(\Omega)}
\leq 
\frac{\sqrt{q(d)}  \|g\|_{\CB_s^2(\Omega)}}{\sqrt{k}},
\]
where \( q(d) = 64d^2 + 48d + 4 \). 
An analogous result also holds if \( \CB_s^2(\Omega) \) is replaced by \( \CB_c^2(\Omega) \).
\end{theorem}
\begin{proof}[\textbf{Proof of Theorem~\ref{mainthm2}}]
The result follows by combining Corollary~\ref{cor:uTepsilon}, Theorem~\ref{thm:baroonnormut+1}, Proposition~\ref{prop:linearofBfunc}, and Theorem~\ref{thmllw21}.
\end{proof}

\section{Conclusion and Future Directions}
In this work, we have established quantitative Barron-space approximation results for second-order elliptic PDEs with homogeneous Dirichlet and Neumann boundary conditions on the unit hypercube. By combining an exponentially convergent Sobolev gradient descent scheme, recursive estimates for the Barron norms of its iterates, and quantitative approximation results for two-layer cosine and ReLU networks, we have shown that the weak solutions can be approximated in the \(H^1(\Omega)\)-norm without incurring the curse of dimensionality, under suitable structural assumptions on the coefficients and the forcing term. Our analysis also provides explicit bounds on the network width and parameters.

Several natural questions remain open. First, the Barron spaces used in this paper are defined through Fourier expansions, and the proof relies essentially on the compatibility between this Fourier structure and the Sobolev gradient. In particular, differentiation, multiplication, and the action of the resolvent \((I-\Delta)^{-1}\) can be controlled explicitly in the corresponding spectral Barron norms. It would therefore be interesting to determine whether the techniques developed here can be extended to PDEs whose coefficients belong only to representational Barron spaces, that is, Barron spaces defined directly through neural-network representations. Such an extension does not appear to be straightforward, since these spaces may not possess the same stability properties under differential operators, multiplication, and elliptic resolvents. Developing an appropriate calculus for representational Barron functions could provide a route toward addressing this question.

Another important open problem is to extend the present framework from the unit hypercube to general bounded domains. The sine and cosine bases used throughout our analysis are closely adapted to both the geometry of the hypercube and its homogeneous boundary conditions. On a general bounded domain, this explicit Fourier structure is no longer available, and the geometry of the boundary must be incorporated into both the Barron-space estimates and the neural-network approximation. One possible direction is to use the eigenfunctions of the Dirichlet Laplacian to define a domain-adapted Barron space, in analogy with~\eqref{eigen}. However, establishing analogues of the neural-network approximation results in Theorems~\ref{thm:nueraltobarron} and~\ref{thmllw21} remains an open problem.

\section*{Acknowledgments}
The work of Z. Chen is supported in part by the National Science Foundation via grant DMS-2509011. Z. Chen thanks Jianfeng Lu and Yulong Lu for helpful discussions. L. Huang gratefully acknowledges  Padmavathi Srinivasan for her continued encouragement and patient guidance during the first two years of graduate study. He also thanks Wenkui Du, Lev Fedorov, Mengxuan Yang, and Shengxuan Zhou for many valuable discussions. We used ChatGPT for minor language editing; all technical content was written and verified by the authors.

\appendix

\section{Omitted Proofs in Sections~\ref{sec:2} and~\ref{sec:proofsofmain}}
\label{sec:A}

\subsection{Proof of Theorem~\ref{bases}}
\label{pfbases}
We denote the domain \( (-1,1)^d \subseteq \mathbb{R}^d \) by \( \widetilde{\Omega} \).

\medskip
\noindent{(1).}
According to~\cite[Theorem~8.20]{Folland99}, the collection in~(1) forms a complete orthogonal system in \(L_{\mathbb{C}}^2(\widetilde{\Omega})\), the space of complex-valued square-integrable functions on \(\widetilde{\Omega}\). Therefore, every function in \( L^2(\Omega) \) admits an \( L^2 \)-expansion in family~(1) with coefficients in \(\mathbb{C}\) because it admits an extension in \( L_{\mathbb{C}}^2(\widetilde{\Omega}) \). Moreover, the extension is not unique, which implies that the \( L^2 \)-expansion is also not unique.

\medskip
\noindent{(2).}
Since the collection in~(1) forms a complete orthogonal system in
\(L_{\mathbb{C}}^2(\widetilde{\Omega})\), and each exponential function
\(e^{i\pi\omega^{\top}x}\) can be expressed as a finite
\(\mathbb{C}\)-linear combination of products of sine and cosine functions,
the collection
\begin{equation}
    \label{SC}
    \left\{ \prod_{i=1}^{d} \operatorname{sc}(\pi k_i x_i) : k=(k_1,\ldots,k_d) \in \mathbb{N}^d,\ x=(x_1,\ldots,x_d) \in \widetilde{\Omega} \right\},
\end{equation}
where \(\operatorname{sc}(\pi k_i x_i)\) denotes either
\(\sin(\pi k_i x_i)\) or \(\cos(\pi k_i x_i)\), with only the cosine choice
allowed when \(k_i=0\), forms a complete orthogonal system in
\(L_{\mathbb{C}}^2(\widetilde{\Omega})\).

Given a function \( g \in L^2(\Omega) \), we define its coordinatewise odd extension by
\[
\widetilde{g}(x) \coloneqq 
\begin{cases}
\operatorname{sgn}(x_1) \cdots \operatorname{sgn}(x_d) g(|x_1|, \dots, |x_d|), & \text{if } x_m \neq 0 \text{ for all } m = 1, \dots, d, \\
0, & \text{otherwise},
\end{cases}
\]
where \(\operatorname{sgn}(t)=1\) for \(t>0\), \(\operatorname{sgn}(t)=0\) for \(t=0\), and \(\operatorname{sgn}(t)=-1\) for \(t<0\).
Since \( \widetilde{g} \) is odd in each coordinate, its expansion in \eqref{SC} contains only sine terms:
\[
\widetilde{g}(x) = \sum_{k \in \mathbb{N}^d} a_k S_k(x) \quad \text{in } L^2(\widetilde{\Omega}).
\]
Restricting this expansion to \( \Omega \) yields a representation of \( g \) in terms of the family~(2). The uniqueness of the expansion is ensured by the orthogonality of distinct functions in this system, and the coefficients are real since 
\[
a_k =2^d \int_{\Omega} g(x) S_k(x) \,\mathrm{d}x \in \mathbb{R}.
\]

\medskip
\noindent{(3).} The proof follows the same argument as in (2) by considering the coordinatewise even extension of \( g \in L^2(\Omega) \), say
\[
\widetilde{g}(x) \coloneqq 
\begin{cases}
g(|x_1|, \dots, |x_d|), & \text{if } x_m \neq 0 \text{ for all } m = 1, \dots, d, \\
0, & \text{otherwise}.
\end{cases}
\]

\medskip
\noindent{(4).} The proof follows the same argument as in (2) by considering an extension of \( g \in L^2(\Omega) \) that is odd in the \( i \)-th and \( j \)-th coordinates and even in the remaining ones, say
\[
\widetilde{g}(x) \coloneqq 
\begin{cases}
\operatorname{sgn}(x_i) \operatorname{sgn}(x_j)g(|x_1|, \dots, |x_d|), & \text{if } x_m \neq 0 \text{ for all } m = 1, \dots, d, \\
0, & \text{otherwise}.
\end{cases}
\]

\subsection{The Poincar\'{e} Constant for the Unit Hypercube}
\label{sec:poincare}
\begin{theorem}
The Poincar\'{e} constant \( C_P \) for the unit hypercube \( \Omega \) is \( 1/\pi \sqrt{d} \); that is,
\[
C_P = \frac{1}{\pi \sqrt{d}}
\]
is the optimal constant such that 
\[
\|\phi\|_{L^2(\Omega)} \leq C_P \|\nabla \phi\|_{L^2(\Omega)}
\quad \text{for all } \phi \in H_0^1(\Omega).
\]
\end{theorem}
\begin{proof}
It is easy to see that
\[
\frac{1}{C_P}
= \inf_{\substack{\phi \in H_0^1(\Omega) \\ \phi \neq 0}}
\frac{\|\nabla \phi\|_{L^2(\Omega)}}{\|\phi\|_{L^2(\Omega)}}.
\]
This quantity is the square root of the smallest eigenvalue of the Dirichlet Laplacian  (cf.~\cite[Theorem~6.5.2]{Evans10}):
\begin{equation}
\label{eigen}
-\Delta u = \lambda u \quad \text{in } \Omega, 
\quad u = 0 \quad \text{on } \partial \Omega.
\end{equation}
It is well known that, for $\Omega = (0,1)^d$, the eigenfunctions of the Dirichlet Laplacian are
\[\left\{S_k(x)=\prod_{i=1}^{d} \sin(\pi k_i x_i) : k=(k_1,\ldots,k_d) \in \mathbb{N}_+^d\right\},\]
and the corresponding eigenvalue for \( S_k(x) \) is \( \pi^2 |k|^2 \) (cf.~\cite[Chapter~6.4.1]{CH89}).  
The smallest eigenvalue of \eqref{eigen} is therefore \( \pi^2 d \).
Hence,
\[
C_P = \frac{1}{\sqrt{\pi^2 d}} 
= \frac{1}{\pi\sqrt{d}}.
\]
\end{proof}

\subsection{Proof of Theorem~\ref{thmllw21}}
\label{genllw21}

As mentioned in Section~\ref{sec3.3}, our strategy is a revision and generalization of the methods in~\cite{LuLuWang21}. 
In particular, Lemma~\ref{LLWlem18} generalizes their Lemma~18, 
Lemma~\ref{LLWprop19} generalizes their Proposition~19, 
and Lemma~\ref{LLWlem16} is simply a restatement of their Lemma~16. 

We let \( I_d \coloneqq [-\sqrt{d},\sqrt{d}] \) throughout the arguments.
\begin{lemma}
\label{LLWlem18}
Suppose \( g \in C^2(I_d) \) and there exists \( B > 0 \) such that 
\[
\| g^{(s)} \|_{L^{\infty}(I_d)} \leq B \quad \text{for all } s =0,1,2,
\]
and there exists \( \alpha \in (-\sqrt{d},\sqrt{d}) \) such that \( g'(\alpha) = 0 \).  
Then, for any positive integer \( m \), there exists a function \( g_m(z) \) defined on \( I_d \) of the form
\begin{equation}
	\label{gmrelu}
g_m(z) = \gamma + \sum_{i=1}^{2m} a_i \operatorname{ReLU}(\varepsilon_i z + b_i),
\end{equation}
where
\begin{equation}
	\label{gmcoeff}
|\gamma| \leq B, \quad |a_i| \leq \frac{2\sqrt{d}B}{m}, \quad \varepsilon_i \in \{\pm 1\}, \quad |b_i| \leq \sqrt{d}, \quad \text{for all } 1 \leq i \leq 2m,
\end{equation}
such that
\begin{equation}
\label{g-gmW}
\| g - g_m \|_{H^{1}(I_d)} \leq \frac{\sqrt{10}d^{5/4}B}{m}.
\end{equation}
\end{lemma}
\begin{proof}
	Let \(\{z_j\}_{0 \leq j \leq 2m}\) be a partition of \(I_d\) with  
\begin{align*}
& z_0 = -\sqrt{d}, \quad z_m = \alpha, \quad z_{2m} = \sqrt{d},\\
& z_{j+1} - z_j = h_1 \coloneqq \frac{\alpha + \sqrt{d}}{m}, \quad j = 0,\ldots, m-1,\\
& z_{j+1} - z_j = h_2 \coloneqq \frac{\sqrt{d} - \alpha}{m}, \quad j = m,\ldots, 2m-1.
\end{align*}
Let \( g_m(z) \) defined on \( I_d \) be the piecewise linear interpolation of \( g(z) \) with respect to \(\{z_j\}_{0 \leq j \leq 2m}\), i.e.,
\[
g_m(z) \coloneqq
\begin{cases}
g(z_{j+1})\dfrac{z - z_j}{h_1} + g(z_j)\dfrac{z_{j+1} - z}{h_1}, & z \in [z_j, z_{j+1}],\ j = 0,\ldots, m-1,\\[1.2ex]
g(z_{j+1})\dfrac{z - z_j}{h_2} + g(z_j)\dfrac{z_{j+1} - z}{h_2}, & z \in [z_j, z_{j+1}],\ j = m,\ldots, 2m-1.
\end{cases}
\] 
Then, according to \cite[Chapter~11]{AG11}, and using \( h_1, h_2 \leq 2\sqrt{d}/m \), we obtain
\begin{equation}
	\label{g-gm}
	\|g - g_m\|_{L^{\infty}(I_d)}
\leq \frac{\max\{h_1,h_2\}^2}{8} \,\|g''\|_{L^{\infty}(I_d)}
\leq \frac{dB}{2m^2}
\leq \frac{dB}{m}.
\end{equation}
We also claim that 
\begin{equation}
\label{g'-gm'}
\|g' - g_m'\|_{L^{\infty}(I_d)} \leq \frac{2\sqrt{d}B}{m}\leq \frac{2dB}{m}.
\end{equation} 
Indeed, if \( z \in [z_j, z_{j+1}] \) for some \( 0 \leq j \leq m-1 \), then by the mean value theorem there exist 
\(\xi, \eta \in (z_j, z_{j+1})\) such that 
\begin{align*}
    |g'(z) - g_m'(z)|& = \left| g'(z) - \frac{g(z_{j+1}) - g(z_j)}{h_1} \right| = |g'(z) - g'(\xi)| \\
    & = |g''(\eta)||z - \xi|\leq |g''(\eta)|h_1 \leq \frac{2\sqrt{d}B}{m}.
\end{align*}
The same bound follows by the same argument when 
\( z \in [z_j, z_{j+1}] \) for some \( m \leq j \leq 2m-1 \).  
This proves \eqref{g'-gm'}.
The estimate~\eqref{g-gmW} then follows by combining~\eqref{g-gm}, \eqref{g'-gm'}, and   
\begin{align*}
    \| g - g_m \|_{H^{1}(I_d)}^2
& = \| g - g_m \|_{L^{2}(I_d)}^2 + \| g' - g_m' \|_{L^{2}(I_d)}^2 \\
& \leq 2\sqrt{d}\|g - g_m\|_{L^{\infty}(I_d)}^2 + 2\sqrt{d}\|g' - g_m'\|_{L^{\infty}(I_d)}^2.
\end{align*}

Now we show that \( g_m(z) \) can be written in the ReLU form~\eqref{gmrelu} and satisfies the coefficient bound in~\eqref{gmcoeff}.  
Indeed, it is straightforward to verify that  
\[
g_m(z) = g(z_m) + \sum_{i=1}^{m} a_i \operatorname{ReLU}(-z + z_i) 
        + \sum_{i=m+1}^{2m} a_i \operatorname{ReLU}(z - z_{i-1}), \quad z\in I_d,
\]
where
\[
a_i =
\begin{cases}
\dfrac{g(z_{i-1}) - 2g(z_i) + g(z_{i+1})}{h_1}, & i = 1,\ldots, m-1,\\[1.0ex]
\dfrac{g(z_{m-1}) - g(z_m)}{h_1}, & i = m,\\[1.0ex]
\dfrac{g(z_{m+1}) - g(z_m)}{h_2}, & i = m+1,\\[1.0ex]
\dfrac{g(z_{i-2}) - 2g(z_{i-1}) + g(z_{i})}{h_2}, & i = m+2,\ldots, 2m.
\end{cases}
\]
It remains to verify that \(|a_i|\) satisfies the upper bound in~\eqref{gmcoeff}.

\medskip
\noindent\textbf{Case \(i = 1,\ldots, m-1\).}  
By the Taylor expansion, there exist \(\eta_1 \in (z_{i-1}, z_i)\) and \(\eta_2 \in (z_i, z_{i+1})\) such that 
\begin{align*}
g(z_{i-1}) &= g(z_i) - g'(z_i)h_1 + \frac{g''(\eta_1)}{2} h_1^2,\\
g(z_{i+1}) &= g(z_i) + g'(z_i)h_1 + \frac{g''(\eta_2)}{2} h_1^2.
\end{align*}
Hence,
\[
|a_i| = \left| \frac{g(z_{i-1}) - 2g(z_i) + g(z_{i+1})}{h_1} \right| 
       = \left|\frac{g''(\eta_1)}{2} h_1 + \frac{g''(\eta_2)}{2} h_1\right| 
       \leq B h_1 
       \leq \frac{2\sqrt{d}B}{m}.
\]
\medskip
\noindent\textbf{Case \(i = m\).}  
By the construction of the partition, \(g'(z_m) = g'(\alpha) = 0\).  
Hence, by the mean value theorem there exist \(\xi, \eta \in (z_{m-1}, z_m)\) such that  
\[
|a_m| = \left| \frac{g(z_{m-1}) - g(z_m)}{h_1} \right| 
       = |g'(\xi)| 
       = |g'(\xi) - g'(z_m)| 
       \leq |g''(\eta)|\,h_1 
       \leq \frac{2\sqrt{d}B}{m}.
\]

\medskip
\noindent\textbf{Case \(i = m+1\).}  
By the same argument as in the case \(i = m\), we obtain the same bound.

\medskip
\noindent\textbf{Case \(i = m+2,\ldots, 2m\).}  
By the same argument as in the case \(1 \leq i \leq m-1\), we obtain the same bound.
\end{proof}

\begin{lemma}
\label{LLWprop19}
For \(B > 0\), define the following families of functions on \(\Omega=(0,1)^d\subseteq \BR^d \) (it is easy to verify that they are all in \(H^1(\Omega)\)):
\begin{align*}
\SF_{\sin}(B) &\coloneqq \left\{ \frac{A}{1+\pi^2|k|^2} \sin(\pi k^{\top}x) : |A| \leq B,\ k \in \mathbb{Z}^d \setminus \{0\} \right\},\\
\SF_{\cos}(B) &\coloneqq \left\{ \frac{A}{1+\pi^2|k|^2} \cos(\pi k^{\top}x) : |A| \leq B,\ k \in \mathbb{Z}^d \setminus \{0\} \right\},\\
\SF_{\operatorname{ReLU}}(B) &\coloneqq \left\{ \gamma + A\operatorname{ReLU}(w^{\top}x + b) : |\gamma| \leq B,\ |A| \leq 4\sqrt{d}B,\ |w| = 1,\ |b| \leq \sqrt{d} \right\}.
\end{align*}
Then the \(H^1(\Omega)\)-closures of the convex hulls of \(\SF_{\sin}(B)\) and \(\SF_{\cos}(B)\) 
are both contained in the \(H^1(\Omega)\)-closure of the convex hull of \(\SF_{\operatorname{ReLU}}(B)\).
\end{lemma}
\begin{proof}
	We only prove the sine function case, as the cosine case follows similarly.  
It suffices to show that  
\[
\frac{A}{1+\pi^2|k|^2} \sin(\pi k^{\top}x) 
\in \text{the $H^1(\Omega)$-closure of } \operatorname{conv}\left(\SF_{\operatorname{ReLU}}(B)\right).
\]

By Lemma~\ref{LLWlem18}, the function  
\(\frac{A}{1+\pi^2|k|^2} \sin(\pi |k| z)\) defined on \(I_d\) can be \(H^1(I_d)\)-approximated by a linear combination of a constant and terms of the form \(\operatorname{ReLU}(\varepsilon z + b)\), with the sum of the absolute values of the coefficients of \(\operatorname{ReLU}(\varepsilon z + b)\) bounded by \(4\sqrt{d}B\). Hence, 
$
\frac{A}{1+\pi^2|k|^2} \sin(\pi |k| z)
$
is contained in the \( H^1(I_d) \)-closure of 
\[
\operatorname{conv}\left\{ \gamma + A\operatorname{ReLU}(\varepsilon z + b) : |\gamma| \leq B,\ |A| \leq 4\sqrt{d}B,\ \varepsilon \in \{\pm 1\},\ |b| \leq \sqrt{d} \right\}.
\]
Since \(|(k/|k|)^{\top}x| \leq |x| \leq \sqrt{d}\), it follows that 
\[
\frac{A}{1+\pi^2|k|^2} \sin(\pi k^{\top}x) 
= \frac{A}{1+\pi^2|k|^2} \sin\!\left(\pi |k|\,(k/|k|)^{\top}x\right)
\]
also lies in the \(H^1(\Omega)\)-closure of 
\begin{equation}
\label{conv}
\operatorname{conv}\left\{ \gamma + A\operatorname{ReLU}\left(\varepsilon (k/|k|)^{\top}x + b\right) : |\gamma| \leq B,\ |A| \leq 4\sqrt{d}B,\ \varepsilon \in \{\pm 1\},\ |b| \leq \sqrt{d} \right\}.
\end{equation}
The result then follows from the fact that \eqref{conv} is contained in \(\operatorname{conv}\left(\SF_{\operatorname{ReLU}}(B)\right)\).
\end{proof}

\begin{lemma}[\cite{P81,Barron93}]
\label{LLWlem16}
Let $(\SG,\|\cdot\|)$ be a subset of a Hilbert space such that the norm of every element in \(\SG\) is bounded by \(B_{\SG} > 0\).  
Suppose \(u\) belongs to the closure of the convex hull of \(\SG\).  
Then, for every positive integer \(k\), there exist \(\{g_i\}_{i=1}^k \subseteq \SG\) and \(\{c_i\}_{i=1}^k \subseteq [0,1]\) with \(\sum_{i=1}^k c_i = 1\) such that 
\[
\left\| u - \sum_{i=1}^k c_i g_i \right\| \leq \frac{B_{\SG}}{\sqrt{k}}.
\]
\end{lemma}

\begin{proof}[\textbf{Proof of Theorem~\ref{thmllw21}}]
	We only prove the case for \(\CB_s^2(\Omega)\), as the case for \(\CB_c^2(\Omega)\) follows similarly.  

Although the definition of an \(s\)-Barron function involves the basis \(\{S_k(x)\}\) 
with \(k = (k_1,\ldots,k_d) \in \mathbb{N}^d\) such that all \(k_i \neq 0\), 
for the purpose of unifying the proof with the \(\CB_c^2(\Omega)\) case, 
we set \(a_k = 0\) for all \(k \in \mathbb{N}^d\) having at least one component \(k_i = 0\) 
in the expansion
\[
g(x) = \sum_{k \in \mathbb{N}^d} a_k S_k(x).
\]
With this convention, we can rewrite
\begin{align*}
g(x) - a_0 =& \sum_{k \in \mathbb{N}^d \setminus \{0\}} a_k S_k(x) 
= \sum_{k \in \mathbb{N}^d \setminus \{0\}} a_k \left( \frac{1}{2^d} 
   \sum_{\epsilon\in \{\pm 1\}^{d} } 
   \theta_{\epsilon} \operatorname{sc}(\pi {(\epsilon\circ k)}^{\top} x) \right) \\
=& \sum_{k \in \mathbb{N}^d \setminus \{0\}} \sum_{\epsilon\in \{\pm 1\}^{d} } 
   \frac{|a_k| (1+\pi^2|\epsilon\circ k|^2)}{2^d A_g} \cdot 
   \frac{\mathrm{sgn}(a_k)\theta_{\epsilon}A_g}{1+\pi^2|\epsilon\circ k|^2} \,\operatorname{sc}(\pi {(\epsilon\circ k)}^{\top} x).
\end{align*}
Here, \[
\theta_{\epsilon}=
\begin{cases}
(-1)^{d/2}\prod_{j=1}^d \epsilon_j, & d \text{ even},\\[1mm]
(-1)^{(d-1)/2}\prod_{j=1}^d \epsilon_j, & d \text{ odd},
\end{cases}
\qquad
\mathrm{sc}=
\begin{cases}
\cos, & d \text{ even},\\
\sin, & d \text{ odd},
\end{cases}
\]
 and   
\[
A_g \coloneqq \sum_{k \in \mathbb{N}^d \setminus \{0\}} |a_k| \left(1 + \pi^2|k|^2\right) \leq \|g\|_{\CB_s^2(\Omega)}.
\]
(In the \(\CB_c^2(\Omega)\) case, only the cosine term appears.)
The case \(A_g=0\) is trivial, so we assume \(A_g>0\). Under this normalization,
\[
\sum_{k \in \mathbb{N}^d \setminus \{0\}} \sum_{\epsilon\in \{\pm 1\}^{d} } 
 \frac{|a_k| (1+\pi^2|\epsilon\circ k|^2)}{2^d A_g} = 1.
\]

Next, we observe that
\[
\|g\|_{H^1(\Omega)}^2
=\sum_{k\in \mathbb{N}^d}|\alpha_k|(1+\pi^2|k|^2)|a_k|^2
\leq \|g\|_{\CB_s^2(\Omega)}^2<\infty,
\]
where \(\alpha_k\) is a constant depending on \(k\) with \(|\alpha_k|\leq 1\).  
This shows that \(g(x) \in H^1(\Omega)\), and hence \(g(x)-a_0 \in H^1(\Omega)\) as well.

Combining the above observations, we conclude that  \(g(x) - a_0\) lies in the \(H^1(\Omega)\)-closure of either \(\operatorname{conv}(\SF_{\cos}(\|g\|_{\CB_s^2(\Omega)}))\) or \(\operatorname{conv}(\SF_{\sin}(\|g\|_{\CB_s^2(\Omega)}))\).  
Hence, by Lemma~\ref{LLWprop19}, it is contained in the \(H^1(\Omega)\)-closure of \(\operatorname{conv}(\SF_{\operatorname{ReLU}}(\|g\|_{\CB_s^2(\Omega)}))\).
Moreover, since \(|a_0| \leq \|g\|_{\CB_s^2(\Omega)}\) (not needed here as \(a_0 = 0\), but required for the \(\CB_c^2(\Omega)\) case; we retain it for consistency of the proofs), it follows that \(g(x)\) lies in the \(H^1(\Omega)\)-closure of the convex hull of  
\begin{align*}
    & \SG\coloneqq \Big\{ \gamma + A\operatorname{ReLU}(w^{\top}x + b) : \\
    &\qquad\qquad\qquad\qquad |\gamma| \leq 2\|g\|_{\CB_s^2(\Omega)},\ |A| \leq 4\sqrt{d}\|g\|_{\CB_s^2(\Omega)},\ |w| = 1,\ |b| \leq \sqrt{d} \Big\}.
\end{align*}
The result then follows from Lemma~\ref{LLWlem16} together with the bound for the \(H^1\)-norm of each \(h \in \SG\):
\begin{align*}
\|h\|_{H^1(\Omega)}^2
&\leq \left( 2\|g\|_{\CB_s^2(\Omega)} + 4\sqrt{d}\|g\|_{\CB_s^2(\Omega)}(\sqrt{d} + \sqrt{d}) \right)^2 + (4\sqrt{d}\|g\|_{\CB_s^2(\Omega)})^2\\
&= (64d^2 + 48d + 4)\|g\|_{\CB_s^2(\Omega)}^2.
\end{align*}
\end{proof}

\addcontentsline{toc}{section}{References}
\bibliographystyle{alpha}
\bibliography{reference.bib}
\vspace{1em}
\noindent 
\small{\textsc{(ZC) Department of Mathematics, Massachusetts Institute of Technology, 77 Massachusetts Avenue, Cambridge, MA 02139, USA}}

\noindent 
\small{\textit{Email address}: \texttt{ziang@mit.edu}}

\vspace{1em}

\noindent 
\small{\textsc{(LH) Department of Mathematics \& Statistics, Boston University, 665 Commonwealth Avenue, Boston, MA 02215, USA}}

\noindent 
\small{\textit{Email address}: \texttt{lqhuang@bu.edu}}
\end{document}